\documentclass[11pt,amsfonts]{article}
\usepackage{graphicx}
\usepackage{latexsym}
\usepackage{amssymb}
\usepackage{amsmath}
\usepackage{enumerate}
\usepackage{color}
\usepackage{layout}
\newtheorem{prop}{Proposition}
\newtheorem{lemma}{Lemma}

\newtheorem{corollary}{Corollary}
\newtheorem{theorem}{Theorem}
\newtheorem{remark}{Remark}
\newtheorem{example}{Example}

\def\real{{\mathord{{\rm I\kern-2.8pt R}}}}        
\def\inte{{\mathord{{\rm I\kern-2.8pt N}}}}

\def\sZZ{{\rm Z\kern-2.8ptem{}Z}}

\def\z{{\mathchoice
		{\sZZ}
		{\sZZ}
		{\rm Z\kern-0.30em{}Z}
		{\rm Z\kern-0.25em{}Z} }}
\def\sQQ{{\kern 0.27em \vrule height1.45ex width0.03em depth0em
		\kern-0.30em \rm Q}}
\def\qu{{\mathchoice
		{\sQQ}
		{\sQQ}
		{\kern 0.225em \vrule height1.05ex width0.025em depth0em \kern-0.25em \rm Q}
		{\kern 0.180em \vrule height0.78ex width0.020em depth0em \kern-0.20em \rm Q}
}}
\def\sCC{{\kern 0.27em \vrule height1.45ex width0.03em depth0em
		\kern-0.30em \rm C}}
\def\complex{{\mathchoice
		{\sCC}
		{\sCC}
		{\kern 0.225em \vrule height1.05ex width0.025em depth0em \kern-0.25em \rm C}
		{\kern 0.180em \vrule height0.78ex width0.020em depth0em \kern-0.20em \rm C}
}}

\newcommand{\ba}{\begin{array}}
	\newcommand{\ea}{\end{array}}
\newcommand{\be}{\begin{equation}}
	\newcommand{\ee}{\end{equation}}
\newcommand{\bea}{\begin{eqnarray}}
	\newcommand{\eea}{\end{eqnarray}}
\newcommand{\beaa}{\begin{eqnarray*}}
	\newcommand{\eeaa}{\end{eqnarray*}}

\newcommand{\eps}{\varepsilon}

\def\z{\zeta}

\font\tenmath=msbm10 \font\sevenmath=msbm7 \font\fivemath=msbm5
\newfam\mathfam \textfont\mathfam=\tenmath
\scriptfont\mathfam=\sevenmath \scriptscriptfont\mathfam=\fivemath

\def \={{\buildrel {\rm (law)} \over =}}

\def\qed{ \hfill \vrule width.25cm height.25cm depth0cm\smallskip}

\newcommand{\basa}{\begin{assumption}}
	\newcommand{\easa}{\end{assumption}}

\newcommand{\bas}{\begin{assum}}
	\newcommand{\eas}{\end{assum}}

\def\liminf{\mathop{\underline{\rm lim}}}

\newcommand{\ignore}[1]{}
\begin{document}
	
	\renewcommand{\thefootnote}{\fnsymbol{footnote}}
	
	\renewcommand{\thefootnote}{\fnsymbol{footnote}}

	\title{Joint convergence in Wiener chaos via  transport hierarchy and Malliavin covariances}
	
	\author{Ciprian A. Tudor \vspace*{0.2in} \\
		CNRS, Universit\'e de Lille \\
	Laboratoire Paul Painlev\'e UMR 8524\\
	F-59655 Villeneuve d'Ascq, France.\vspace*{0.1in}\\
	\quad ciprian.tudor@univ-lille.fr\\
}

	\maketitle
	
	\begin{abstract}
		We study the joint convergence in distribution of a sequence 
		$X_N = I_p(f_N)$ of multiple Wiener--It\^o integrals of order 
		$p\geq 2$ that converges to a Gaussian limit $Z\sim N(0,\sigma^2)$, 
		together with another sequence $Y_N = I_q(g_N)$ converging in law. 
		The central finding is that the joint convergence of $(X_N, Y_N)$ 
		is completely governed by the asymptotic behavior of the iterated 
		Malliavin covariances $Y_{r+1,N} = \langle DX_N, DY_{r,N}\rangle_H$, 
		$r\geq 0$: joint convergence holds as soon as these covariances 
		converge jointly with $Y_N$, and the structure of the limiting 
		distribution is then explicitly determined by their limits. Moreover, 
		the convergence of the  Malliavin covariances  is 
		necessary for joint convergence, as shown by a counterexample.
		
		When $q<p$, the sequence $X_N$ is asymptotically independent 
		of any $Y\in L^2(\Omega)$, a result which strengthens the stable 
		convergence results in  \cite{T3} and extends the multidimensional Fourth Moment Theorem \cite{PeTu}. When $q \geq  p$, genuine asymptotic dependence appears and 
		its structure depends critically on the ratio $q/p$. Writing 
		$q = ap + r'$ with $0\leq r' < p$, the iterated Malliavin 
		covariances form a transport hierarchy of depth $a$ that terminates 
		in both the non-critical regime $ap < q < (a+1)p$ and the critical 
		regime $q = ap$, but with different structures: the hierarchy is 
		nilpotent in the non-critical case and recurrent in the critical one, 
		due to the non-vanishing limit $\rho_a = \lim_N \mathbf{E}[Y_{a,N}]$. 
		In both cases, the limiting characteristic function admits an explicit 
		series representation whose coefficients are determined by a simple 
		recursion. Under exponential moment assumptions, the series closes 
		in closed form, and the two regimes differ by exactly one additional 
		factor that appears only in the critical case.
	\end{abstract}

		\vskip0.3cm
	
	{\bf 2010 AMS Classification Numbers:}  60F05,60G15,60H05,60H07.

	\vskip0.3cm
	
	
	\section{Introduction}
	
Let $(H, \langle \cdot, \cdot\rangle _{H})$ be a real and separable Hilbert space and let $(W(h), h\in H)$ be an isonormal process. For $p\geq 0$, let $I_{p} $ be the multiple integral of order $p$ with respect to $W$. 	Let $(X_N = I_p(f_N), N\geq 1)$ be a sequence of multiple Wiener--It\^o 
	integrals of order $p\geq 2$ that converges in distribution to a Gaussian 
	limit $Z\sim N(0,\sigma^2)$. The central question of this paper 
	is: \emph{given another sequence $(Y_N, N\geq 1)$ converging in law, when 
		does the pair $(X_N, Y_N)$ converge jointly, and what is the law of the 
		limit?}
	
	This question is fundamentally different from the one-dimensional problem. 
	While the Fourth Moment Theorem of Nualart and Peccati \cite{NuPe} gives 
	a clean criterion for $X_N \to_{N \to \infty}  Z$ in distribution, the joint convergence 
	of $(X_N, Y_N)$ is a genuinely more subtle problem whose answer depends 
	in a delicate way on the \emph{Malliavin covariance structure} between 
	$X_N$ and $Y_N$.
	
	\medskip
	
	\noindent\textbf{The main message.} 
	The central finding of this paper can be stated informally as follows:
	
	\begin{quote}
		\emph{The joint convergence of $(X_N, Y_N)$ holds as soon as the iterated 
			Malliavin covariances
			$$Y_{r+1,N} := \langle DX_N, DY_{r,N}\rangle_H, \quad r= 0, 1, 2,\ldots,$$
			converge jointly with $Y_N$. Moreover, when this condition holds, the 
			full structure of the joint limit $(Z,Y)$ is explicitly determined by 
			the limits of these covariances. On the other hand, the convergence of 
			the Malliavin covariances IS necessary for the joint 
			convergence of $(X_N, Y_N)$, as demonstrated by Example~\ref{ex1}.}
	\end{quote}
	
	This is a sharp and complete picture: the iterated Malliavin covariances 
	are not merely a technical device but the \emph{essential invariants} of 
	the joint convergence problem.
	
	\medskip
	
	\noindent\textbf{Why the Malliavin covariance matters.}
	The heuristic is transparent from the integration by parts formula. 
	For the characteristic function 
	$\varphi_N(u,v) = \mathbf{E}[e^{iuX_N + ivY_N}]$, differentiation with respect 
	to $u$ and the Malliavin integration by parts give
	\begin{equation}\label{ibp_intro}
		\frac{\partial \varphi_N}{\partial u}(u,v) 
		= -\sigma^2 u\varphi_N(u,v) 
		- \frac{v}{p} \mathbf{E}\!\left[Y_{1,N}\, e^{iuX_N+ivY_N}\right] + \varepsilon_N  (u,v),
	\end{equation}
	where $\varepsilon_N (u,v)\to_{N \to \infty} 0$  uniformly in $(u,v)$, by the Fourth Moment Theorem. The evolution 
	of $\varphi_N$ in $u$ is thus driven by the weighted expectation 
	$\mathbf{E}[Y_{1,N} e^{iuX_N+ivY_N}]$. If $Y_N$ lives in a chaos of order 
	$q\leq p$, the Malliavin covariance $Y_{1,N}$ vanishes asymptotically 
	(Lemma~\ref{ll4}), the equation \eqref{ibp_intro} decouples, and $Z$ 
	and $Y$ are independent in the limit. When $q>p$, the covariance does 
	not vanish and one must track its asymptotic behavior. The quantity 
	$\mathbf{E}[Y_{1,N} e^{iuX_N+ivY_N}]$ then satisfies its own evolution equation 
	involving $Y_{2,N} = \langle DX_N, DY_{1,N}\rangle_H$, and so on: one 
	is led to a \emph{transport hierarchy} on the weighted characteristic 
	functions
	\begin{equation*}
		G_{m,N}(u,v) := \mathbf{E}\!\left[\prod_{r=1}^{a} Y_{r,N}^{m_r}\, 
		e^{iuX_N+ivY_N}\right], \qquad m\in\mathbb{N}^{a},
	\end{equation*}
	which is the key structure of this paper. The hierarchy terminates 
	at depth $a$ because $Y_{a+1,N}\to_{N \to \infty} 0$   in $L^ {2}(\Omega)$ in the non-critical regime 
	$ap<q<(a+1)p$ 
	and $Y_{a,N}$ converges to a deterministic constant $\rho_a$ in the 
	critical regime $q=ap$ (Proposition  \ref{prop:iterated-covariances}).
	
	\medskip
	
\noindent\textbf{The structure of the limit.}
The key analytical tool is the \emph{transport hierarchy}: after 
renormalization, the weighted characteristic functions $G_{m,N}$ 
satisfy a recursive system of ODEs in $u$, driven by the transport 
of mass between multi-indices $m$ through \\forward steps (corresponding 
to the Malliavin covariance $Y_{1,N}$), shift steps (corresponding to 
higher-order covariances $Y_{2,N},\ldots$), and, in the critical regime, 
a backward coupling driven by the constant $\rho_a$. Passing to the 
limit and solving this transport system yields an explicit series 
representation for the limiting characteristic function $\varphi(u,v)$, 
whose coefficients are determined by a simple recursion. Under the 
additional assumption that the limiting covariances (i.e. the limits of $Y_{r, N}, r=1,...,a$ given by (\ref{yr})) $Y_1,\ldots,Y_a$ 
admit exponential moments, the series closes in closed form.

To illustrate, consider first the case $q=p$ (Theorem~\ref{tt2}): 
the Malliavin covariance converges to a constant $p\rho$ and the 
transport equation is a simple linear ODE, giving
\begin{equation*}
	\varphi(u,v) = e^{-\sigma^2 u^2/2 - uv\rho}\,\mathbf{E}\!\left[e^{ivY}\right],
\end{equation*}
which shows that $Z$ and $Y$ are linearly correlated with coefficient 
$\rho$. For $p<q<2p$ (the simplest non-trivial case, Corollary~\ref{cor22} 
with $a=1$), the limiting covariance $Y_1$ is random and the closed 
form under exponential moments gives
\begin{equation*}
	\varphi(u,v) = e^{-\sigma^2 u^2/2}\,\mathbf{E}\!\left[e^{ivY - \frac{uv}{p}Y_1}\right],
\end{equation*}
showing that the Gaussian characteristic function of $Z$ is twisted by 
the random limit Malliavin covariance $Y_1$. In general, each successive 
iterated covariance $Y_r$ enters at order $u^r$ in the exponent, 
reflecting deeper layers of dependence between $X_N$ and $Y_N$.
	
	\medskip
	
\noindent\textbf{Relation to the literature.}
The asymptotic independence of Wiener chaos sequences has been studied 
in several important works. The case $q< p$ (asymptotic independence) 
was treated quantitatively in  \cite{T3} using a multidimensional 
Stein--Malliavin method which extends the findings in \cite{Pi}; our 
Theorem~\ref{tt1} and Proposition~\ref{prop:indep} sharpen these results 
by removing any assumption on the vanishing of the covariance. The case 
when $ Y_{N}$  is itself  asymptotically Gaussian in a fixed Wiener chaos is the setting of the Peccati--Tudor theorem \cite{PeTu}, which 
states that a sequence of vectors of multiple Wiener--It\^o integrals 
of fixed orders converges jointly in distribution to a Gaussian vector 
if and only if each component converges marginally. Our Theorem~\ref{tt2} 
extends this framework in two directions: it does not require $Y_N$ to 
converge to a Gaussian limit, and it identifies the limiting \\
characteristic function explicitly in terms of the asymptotic covariance 
$\rho=\lim_{N\to \infty} \mathbf{E}[X_NY_N]$. The reference \cite{NR} studied 
asymptotic independence of multiple integrals under certain contraction 
conditions, and   \cite{NNP} established 
strong asymptotic independence on Wiener chaos. In \cite{HMP} the authors studied limit distributions for polynomials with 
i.i.d.\ entries, uncovering related universality phenomena with arguments based on asymptotic independence. The case 
$q>p$, which is the main contribution of this paper, goes entirely 
beyond the Peccati--Tudor framework since $Y_N$ does not converge to 
a Gaussian limit and the joint limit $(Z,Y)$ is not a Gaussian vector: 
to our knowledge, no explicit formula for the joint characteristic 
function and no complete characterization of joint convergence via 
the iterated Malliavin covariances was previously known. The transport 
hierarchy approach, which reduces the problem to a triangular or 
recurrent ODE system on the weighted characteristic functions $G_{m,N}$, 
is also new and may be of independent interest beyond the Wiener chaos 
setting.
	
	\medskip

\noindent\textbf{Organization.}
The paper is organized as follows. 
Section~\ref{sec:stable} develops the case $q\leq p$: 
after collecting preliminary results on contractions and 
Malliavin derivatives, we establish stable convergence 
and asymptotic independence of an asymptotically Gaussian sequences in a fixed chaos from the whole $L^ {2}(\Omega)$ space (Theorems~\ref{tt1} 
and~\ref{tt3}, Proposition~\ref{pp1}), and treat the 
case $q=p$ where linear correlation may appear 
(Theorem~\ref{tt2}). Section~\ref{sec:general} contains 
the main results. We begin with the key structural result 
on the iterated Malliavin covariances 
(Proposition~\ref{prop:iterated-covariances})  and then we prove the main results.  We treat the non-critical regime $ap<q<(a+1)p$  in Theorem~\ref{tt6} and 
Corollary~\ref{cor22}. The critical regime $q=ap$ is 
treated in Theorem~\ref{tt7} and Corollary~\ref{cor:tt7}. 
Section~\ref{sec:structure} gives a complementary 
result to show that 
the dependence of $Y$ on $Z$ is entirely encoded in a 
polynomial in $Z$ with coefficients determined by the 
iterated Malliavin covariances, plus an independent 
remainder. Section~\ref{sec:applications} contains some 
examples and applications  illustrating the main results. The appendices 
collect background material on Wiener chaos and Malliavin 
calculus, and the proofs of results deferred from the 
main text.

	\section{Stable convergence and asymptotic independence}\label{sec:stable} 
	
Let $p\geq 2$ be an integer. We show that an asymptotically 
Gaussian sequence $X_N = I_p(f_N)$ is asymptotically independent of any 
$Y \in L^2(\Omega)$ (Theorem~\ref{tt1}), a result which can be interpreted 
as stable convergence of $X_N$ with respect to the full $\sigma$-algebra 
of the probability space. We then extend this to sequences $\mathbb{Y}_N$ 
whose components live in a single Wiener chaos of order strictly less than 
$p$ (Proposition~\ref{prop:indep}), in a finite sum of Wiener chaoses 
(Theorem~\ref{tt2}), and in an infinite chaos expansion 
(Theorem~\ref{tt3}). In all these cases the limiting pair $(Z, \mathbb{Y})$ 
has independent components, except when $q=p$ where a nonzero asymptotic 
covariance $\rho = \lim_{N\to \infty} \mathbf{E}[X_N Y_N]$ may produce linear correlation between 
$Z$ and $Y$ in the limit. The preliminary lemmas needed throughout the paper 
are collected in Section~\ref{sec:prelim}.
	
	\subsection{Some preliminary results}\label{sec:prelim}
	We gather in this section several preliminary results concerning sequences of random variables in Wiener chaos that will be needed throughout the paper. 	Let $ (H, \langle \cdot, \cdot \rangle_{H})$ be a real and separable Hilbert space.   The central assumption is that a sequence $(X_N = I_p(f_N)), f_{N}\in H ^ {\odot p}$ in the $p$th Wiener chaos converges in distribution to a Gaussian random variable, which by the Fourth Moment Theorem is equivalent to the asymptotic vanishing of the contraction $f_N \otimes_{r} f_N$
	in $H^{\otimes 2p-2r}$ for every $r=1,...,p-1.$   Starting from this condition, we derive a series of consequences concerning the asymptotic behavior of inner products and contractions involving $f_N$. These results, while largely technical in nature, play a fundamental role in the proofs of the main theorems and clarify the precise sense in which the kernels $f_N$
	become asymptotically orthogonal to any fixed or convergent sequence in the relevant tensor product spaces.
	
		\begin{lemma} \label{ll2}
		Let $ (f_{N}, N\geq 1)$ be a sequence  in $ H ^{\odot p}$ with $p\geq 2$. Assume	\begin{equation}\label{a1}
			f_{N} \otimes _{1} f_{N}\to _{N \to \infty} 0  
			\mbox{ in } H ^{\otimes 2p-2}.
		\end{equation} and that there exists $M>0$ such that 
		\begin{equation}\label{a2}
			\sup_{N\geq1 } \Vert f_{N} \Vert ^{2} _{ H^{\otimes p}}\leq M.
		\end{equation}
		Then 
		
		\begin{enumerate}
			\item 	For every $ g\in H ^{\odot p}$ we have 
			
			\begin{equation*}
				\langle f_{N}, g\rangle _{ H ^{\otimes p}} \to _{N \to \infty} 0.
			\end{equation*}
			\item Let $ (g_{N}, N\geq 1)$  be a sequence in $ H ^{\odot p}$ which converges in $ H ^{\otimes p}$. Then 
			\begin{equation*}
				\langle f_{N}, g_{N}\rangle _{ H ^{\otimes p}}\to _{N \to \infty }0.
			\end{equation*}
			
		\end{enumerate}
		
	\end{lemma}
	{\bf Proof: } The proof is relegated to the Appendix. \qed

	The results in Lemma \ref{ll1} and Lemma \ref{ll2} are satisfied for kernels of a sequence in the $p$th Wiener chaos which converges to a normal law.  By $\xrightarrow[N\to\infty]{(d)} $ we denote the convergence in distribution.  
	\begin{corollary}\label{cor1}
		Let $p\geq 2$ and 	let $ (X_{N}= I_{p} (f_{N}), N\geq 1)$ with $ f_{N} \in H ^{\odot p}$. Assume 
		\begin{equation}\label{c1}
			X_{N} \xrightarrow[N\to \infty]{(d)} N(0, \sigma ^{2}),
		\end{equation}
		where $\sigma ^{2}>0$. Then
		\begin{enumerate}
			\item  For every $ g\in H ^{\odot p}$, we have
			\begin{equation*}
				\langle f_{N}, g\rangle _{ H ^{\otimes p}} \to _{ N \to \infty} 0.
			\end{equation*}
			\item For every sequence $(g_{N}, N\geq 1)\subset H ^{\odot p}$ which converges in $ H ^{\otimes p}$, 
			\begin{equation*}
				\langle f_{N}, g_{N}\rangle _{ H ^{\otimes p}} \to _{ N \to \infty} 0.
			\end{equation*}
		\end{enumerate}
	\end{corollary}
	{\bf Proof: } The result is a consequence of Lemma \ref{ll1} and Lemma \ref{ll2}, since by the  Fourth Moment Theorem (point 3. in Theorem \ref{4mom})  and (\ref{20m-1}), the sequence $ (f_{N}, N\geq1)$ satisfies (\ref{a1}) and (\ref{a2}). \qed

	Let us also recall two useful lemmas from \cite{T3}. 
	
	\begin{lemma}\label{ll3}
		Let $p\geq 2$ and $q\geq 1$ two integer numbers. Let $ (X_{N}= I_{p} (f_{N}), N\geq 1)$ with $f_{N} \in H ^{\odot p}$ for all $N\geq 1$. Assume (\ref{c1}). Then
		
		\begin{enumerate}
			\item for every $ g\in H ^{\odot q}$, 
			\begin{equation*}
				\Vert f_{N} \otimes _{r} g\Vert _{ H ^{\otimes p+q-2r}} \to _{N \to \infty }0 \mbox{ for every } \begin{cases}
					r=1,..., p\wedge q, \mbox{ if } p\not =q\\
					r=1,..., p-1 \mbox{ if } p=q. 
				\end{cases}
			\end{equation*}
			
			\item  for every sequence $(g_{N}, N\geq 1)\subset H ^{\otimes q}$ which converges in $ H ^{\otimes q}$, we have 
			\begin{equation*}
				\Vert f_{N} \otimes _{r} g_{N}\Vert _{ H ^{\otimes p+q-2r}} \to _{N \to \infty }0 \mbox{ for every } \begin{cases}
					r=1,..., p\wedge q, \mbox{ if } p\not =q\\
					r=1,..., p-1 \mbox{ if } p=q. 
				\end{cases}
			\end{equation*}
			\item  for every bounded sequence $(g_{N}, N\geq 1)\subset H ^{\otimes q}$ with $q\leq p$ we have 
			\begin{equation*}
				\Vert f_{N} \otimes _{r} g_{N}\Vert _{ H ^{\otimes p+q-2r}} \to _{N \to \infty}0 \mbox{ for every } \begin{cases}
					r=1,...,q, \mbox{ if } p\not =q\\
					r=1,..., p-1 \mbox{ if } p=q. 
				\end{cases}
			\end{equation*}
		\end{enumerate}
	\end{lemma}

	\subsection{Stable convergence and asymptotic independence } \label{subsec:stable}
	In this section, we investigate the joint convergence in distribution of a random sequence $(X_N, Y_N)$, where $X_N = I_p(f_N)$ is asymptotically Gaussian and $Y_N$ is an arbitrary square-integrable random variable or, more generally, a sequence belonging to a finite sum of Wiener chaoses. We first establish a stable convergence result showing that $X_N
	$ converges jointly with any fixed $Y \in L^2(\Omega)$, with $
	Z $ and $Y$ independent in the limit. This result, which relies on the density of smooth functionals in $L^2(\Omega)$ combined with a standard approximation argument, extends to the case where  $Y_N$ is itself a sequence converging in $L^2(\Omega)$
	or in law. The key technical ingredient is the vanishing of the Malliavin covariance $\langle DX_N, DY_N \rangle_H$ in the limit, which, as we shall see, is the precise mechanism behind asymptotic independence in this setting.
	
	We first show that an asymptotically Gaussian sequence in a Wiener chaos of order at least  two is asymptotically independent of any variable in $ L^ {2}(\Omega). $
\begin{theorem}\label{tt1}
	Let $p\geq 2$ and let $(X_N = I_p(f_N), N\geq 1)$ with $f_N\in H^{\odot p}$ 
	for all $N\geq 1$. Assume \eqref{c1}. Then, for every $Y\in L^2(\Omega)$,
	\begin{equation*}
		(X_N, Y) \xrightarrow[N\to\infty]{(d)} (Z, Y),
	\end{equation*}
	where $Z\sim N(0,\sigma^2)$ and $Z$, $Y$ are independent.
\end{theorem}
{\bf Proof: }$Y$ admits the chaos expansion $Y = \sum_{n=0}^\infty I_n(g_n)$ with 
$g_n\in H^{\odot n}$. By Corollary~\ref{cor1}, $\langle f_N, g_p
\rangle_{H^{\otimes p}}\to 0$ as $N\to\infty$. Since 
$\mathbf{E}[X_N Y] = p!\,\langle f_N, g_p\rangle_{H^{\otimes p}}$,
we deduce that
\begin{equation*}
	\mathbf{E}[X_N Y] \to_{N\to\infty} 0.
\end{equation*}

Assume first that $Y\in\mathbb{D}^{1,2}$. Then the result is a direct 
consequence of Corollary~2 in \cite{T3}.

Assume now that $Y\in L^2(\Omega)$. Let $\mathcal{S}$ be the class of 
smooth functionals of the form
\begin{equation*}
	F = f(W(h_1),\ldots,W(h_n)) \quad\text{with } n\geq 1,\ 
	f\in C_b^\infty(\mathbb{R}^n),\ h_i\in H.
\end{equation*}
Then $\mathcal{S}$ is dense in $L^2(\Omega)$ and $\mathcal{S}\subset
\mathbb{D}^{1,2}$, so there exists a sequence $(Y_k, k\geq 1)$ in 
$\mathcal{S}$ such that
\begin{equation}\label{2f-3}
	Y_k \to_{k\to\infty} Y \quad\text{in } L^2(\Omega).
\end{equation}
We write, for every $u\in\mathbb{R}$ and $v\in\mathbb{R}$,
\begin{eqnarray*}
	\left|\mathbf{E}\!\left[e^{iuX_N+ivY}\right] 
	- \mathbf{E}\!\left[e^{iuZ+ivY}\right]\right|
	&\leq&
	\left|\mathbf{E}\!\left[e^{iuX_N+ivY}\right] 
	- \mathbf{E}\!\left[e^{iuX_N+ivY_k}\right]\right|\\
	&&+\,\left|\mathbf{E}\!\left[e^{iuX_N+ivY_k}\right] 
	- \mathbf{E}\!\left[e^{iuZ+ivY_k}\right]\right|\\
	&&+\,\left|\mathbf{E}\!\left[e^{iuZ+ivY_k}\right] 
	- \mathbf{E}\!\left[e^{iuZ+ivY}\right]\right|\\
	&=:& A_{N,k} + B_{N,k} + C_k.
\end{eqnarray*}
Using $|e^{ix}-e^{iy}|\leq|x-y|$ and \eqref{2f-3},
\begin{equation*}
	A_{N,k} + C_k \leq 2|v|\sqrt{\mathbf{E}\!\left[(Y_k-Y)^2\right]}
	\to_{k\to\infty} 0,
\end{equation*}
uniformly in $N$. On the other hand, since $Y_k\in\mathbb{D}^{1,2}$ 
for every $k\geq 1$, the first part of the proof gives
\begin{equation*}
	B_{N,k} \to_{N\to\infty} 0 \quad\text{for every fixed } k\geq 1.
\end{equation*}
We conclude by taking first $N\to\infty$ (so that $B_{N,k}\to 0$ for 
each fixed $k$) and then $k\to\infty$ (so that $A_{N,k}+C_k\to 0$ 
uniformly in $N$): for every $\varepsilon>0$, choose $k_0$ large 
enough so that $A_{N,k_0}+C_{k_0}\leq\varepsilon/2$ for all $N$, 
then choose $N_0$ large enough so that $B_{N,k_0}\leq\varepsilon/2$ 
for all $N\geq N_0$. Then for $N\geq N_0$,
\begin{equation*}
	\left|\mathbf{E}\!\left[e^{iuX_N+ivY}\right] 
	- \mathbf{E}\!\left[e^{iuZ+ivY}\right]\right| \leq \varepsilon,
\end{equation*}
which gives the conclusion since $\mathbf{E}[e^{iuZ+ivY}]
=e^{-\sigma^2 u^2/2}\mathbf{E}[e^{ivY}]$ is the characteristic 
function of $(Z,Y)$ with $Z$ and $Y$ independent.\qed

	\begin{remark}{\rm 
			
			Theorem \ref{tt1} strengthens the stable convergence result of Corollary 2 in 
			\cite{T3}, where asymptotic independence was established under the additional 
			assumption that the $\mathbf{E}[X_{N}Y]$  vanishes 
			asymptotically. Here, no such assumption is needed, due to Corollary \ref{cor1}: the asymptotic Gaussianity of 
			$X_N$ alone, via the Fourth Moment Theorem, is sufficient to guarantee the joint 
			convergence of $X_N, Y$ for any $Y \in L^2(\Omega)$, with $Z$ and $Y$ independent 
			in the limit. This result can also be interpreted as a stable convergence statement: 
			the sequence $(X_N)$ converges stably in law to $Z \sim N(0, \sigma^2)$, 
			meaning that the convergence in distribution holds jointly with any random variable 
			defined on the same probability space, which is a strictly stronger mode of convergence 
			than mere convergence in distribution.}
	\end{remark}

	We next analyze the asymptotic independence of $X_{N} $ in the $p$th Wiener chaos which is asymptotically normal and another sequence in Wiener chaos. We first derive some result concerning the behavior of the Malliavin derivatives.

	\begin{lemma}\label{ll4}
		Let $ (X_{N}= I_{p} (f_{N}), N\geq 1)$ with $p\geq 2$ satisfying (\ref{c1}). 
		
		\begin{enumerate}
			\item Let $ (Y_{N}=I_{p}(g_{N}), N\geq 1) $ be a sequence of random variables which converges in $ L ^{2}(\Omega)$.  Assume that there exists $\rho \in \mathbb{R}$ such that
			\begin{equation}\label{c3}
				\mathbf{E}[ X_{N} Y_{N} ]\to _{N \to \infty} \rho. 
			\end{equation}
			Then 
			\begin{equation*}
				\langle  DX_{N}, DY_{N} \rangle _{ H ^{\otimes p}} \to _{N \to \infty} p\rho \mbox{ in } L ^{2}(\Omega). 
			\end{equation*}
			\item Let $ (Y_{N}=I_{q}(g_{N}), N\geq 1) $ be a sequence of random variables in the $q$ th Wiener chaos such that $q<p$. Assume that $(Y_{N}, N\geq 1)$ is bounded in $ L ^ {2}(\Omega)$. Then 
			\begin{equation*}
				\langle  DX_{N}, DY_{N} \rangle _{ H ^{\otimes p}} \to _{N \to \infty} 0 \mbox{ in } L ^{2} (\Omega). 
			\end{equation*}
			
		\end{enumerate}
	\end{lemma}
	{\bf Proof: } See the Appendix. \qed  
	
	We deduce that an asymptotically Gaussian sequence in a Wiener chaos is asymptotically independent of any $d$-dimensional sequence with component in Wiener chaos of order strictly less than $p$.

	\begin{prop}
		\label{prop:indep}
		Let $p \geq 2$ and let $(X_N = I_p(f_N), N \geq 1)$ with $f_N \in H^{\odot p}$ 
		for all $N \geq 1$. Assume that $(X_N, N \geq 1)$ satisfies \eqref{c1}. Let 
		$(\mathbb{Y}_N = (Y_{1,N}, \ldots, Y_{d,N}), N \geq 1)$ be a $d$-dimensional random sequence 
		such that for every $j = 1, \ldots, d$, $Y_{j,N}$ belongs to the $q_j$-th Wiener 
		chaos with $q_j < p$. Assume that
		\begin{equation*}
			\mathbb{Y}_N \xrightarrow[N\to\infty]{(d)} \mathbb{Y} = (Y_1, \ldots, Y_d).
		\end{equation*}
		Then
		\begin{equation*}
			(X_N, \mathbb{Y}_N) \xrightarrow[N\to\infty]{(d)} (Z, \mathbb{Y}),
		\end{equation*}
		where $Z \sim N(0, \sigma^2)$ and $Z$ is independent of $\mathbb{Y}$.
	\end{prop}
{\bf Proof: }The random sequence $((X_N, \mathbb{Y}_N), N\geq 1)$ is bounded in 
	$L^2(\Omega)$ by \eqref{20m-1}, so it is tight and relatively compact 
	by Prokhorov's theorem. To get the conclusion, it suffices to show 
	that any subsequence of $(X_N, \mathbb{Y}_N)$ that converges in 
	distribution has limit $(Z, \mathbb{Y})$ with characteristic function
	\begin{equation}\label{phi_indep}
		\varphi(u,v) = e^{-\frac{\sigma^2 u^2}{2}}\,
		\mathbf{E}\!\left[e^{i\langle v,\mathbb{Y}\rangle}\right],
		\qquad u\in\mathbb{R},\ v\in\mathbb{R}^d,
	\end{equation}
	where $\langle \cdot, \cdot \rangle$ is the Euclidean scalar product. For simplicity, we also denote by $(X_N, \mathbb{Y}_N)$ such a 
	subsequence and let $\varphi_N$ be its characteristic function. 
	By standard arguments, we have, for every $u\in\mathbb{R}$, 
	$v\in\mathbb{R}^d$,
	\begin{equation}\label{2f-7}
		\frac{\partial\varphi_N}{\partial u}(u,v) 
		\to_{N\to\infty} 
		\frac{\partial\varphi_{(Z,\mathbb{Y})}}{\partial u}(u,v).
	\end{equation}
	We now show that the limiting characteristic function factorizes. 
	Using the Malliavin integration by parts \eqref{dua} with 
	$F=e^{iuX_N+i\langle v,\mathbb{Y}_N\rangle}$ and $u=(-L)^{-1}DX_N$, 
	and recalling that $D(-L)^{-1}X_N = \frac{1}{p}DX_N$ for 
	$X_N\in \mathcal{H}_p$, we obtain
	\begin{eqnarray}
		\frac{\partial\varphi_N}{\partial u}(u,v) 
		&=& i\,\mathbf{E}\!\left[X_N e^{iuX_N+i\langle v,\mathbb{Y}_N\rangle}
		\right] \nonumber\\
		&=& -\frac{u}{p}\,\mathbf{E}\!\left[\|DX_N\|_H^2\,
		e^{iuX_N+i\langle v,\mathbb{Y}_N\rangle}\right] \nonumber\\
		&&-\frac{1}{p}\sum_{j=1}^d v_j\,\mathbf{E}\!\left[
		\langle DX_N, DY_{j,N}\rangle_H\,
		e^{iuX_N+i\langle v,\mathbb{Y}_N\rangle}\right].
		\label{2f-5}
	\end{eqnarray}
	Since $Y_{j,N}=I_{q_j}(g_{j,N})$ converges in law in a fixed Wiener 
	chaos of order $q_j < p$, it is bounded in $L^2(\Omega)$ by 
	\eqref{20m-1}. Hence Lemma~\ref{ll4}, point 2 applies and gives
	\begin{equation*}
		\langle DX_N, DY_{j,N}\rangle_H \to_{N\to\infty} 0 
		\quad\text{in } L^2(\Omega),
	\end{equation*}
	so the second term in \eqref{2f-5} vanishes as $N\to\infty$. For 
	the first term, by the Fourth Moment Theorem (Theorem \ref{4mom}), 
	$\frac{1}{p}\|DX_N\|_H^2\to\sigma^2$ in $L^2(\Omega)$ as $N\to \infty$, so
	\begin{eqnarray}
		&&\left|\mathbf{E}\!\left[\frac{1}{p}\|DX_N\|_H^2\,
		e^{iuX_N+i\langle v,\mathbb{Y}_N\rangle}\right]
		-\sigma^2\,\mathbf{E}\!\left[e^{iuZ+i\langle v,\mathbb{Y}\rangle}
		\right]\right|\nonumber\\
		&\leq& \mathbf{E}\!\left[\left|\frac{1}{p}\|DX_N\|_H^2
		-\sigma^2\right|\right]
		+\sigma^2\left|\mathbf{E}\!\left[e^{iuX_N+i\langle v,
			\mathbb{Y}_N\rangle}\right]
		-\mathbf{E}\!\left[e^{iuZ+i\langle v,\mathbb{Y}\rangle}
		\right]\right|\nonumber\\
		&\to_{N\to\infty}& 0,\label{22m-2}
	\end{eqnarray}
	where the first term goes to zero by the Fourth Moment Theorem and 
	the second by the convergence in law of $(X_N,\mathbb{Y}_N)$ to 
	$(Z,\mathbb{Y})$. Taking the limit in \eqref{2f-5} as $N\to\infty$,
	\begin{equation*}
		\frac{\partial\varphi_{(Z,\mathbb{Y})}}{\partial u}(u,v) 
		= -\sigma^2 u\,\varphi_{(Z,\mathbb{Y})}(u,v), \mbox{ for every }u\in \mathbb{R},, v\in \mathbb{R} ^ {d}.
	\end{equation*}
	Since $\varphi_{(Z,\mathbb{Y})}(0,v)=\mathbf{E}\!\left[e^{i\langle v,
		\mathbb{Y}\rangle}\right]$, the unique solution is
	\begin{equation*}
		\varphi_{(Z,\mathbb{Y})}(u,v) 
		= e^{-\frac{\sigma^2 u^2}{2}}\,
		\mathbf{E}\!\left[e^{i\langle v,\mathbb{Y}\rangle}\right]
		= \varphi_Z(u)\cdot\varphi_{\mathbb{Y}}(v),
	\end{equation*}
	which is the characteristic function of $(Z,\mathbb{Y})$ with 
	$Z\sim N(0,\sigma^2)$ independent of $\mathbb{Y}$.
\qed

\begin{remark}{\rm
		The paper \cite{HMP}  studies limit distributions 
		for polynomials in i.i.d.\  Gaussain entries and establishes, in that setting, results 
		on asymptotic independence that are analogous in spirit to those of 
		Section~\ref{subsec:stable}. On the other hand, they 
		do not address the problem  of joint convergence when the Malliavin covariance does not vanish, 
		and the explicit description of the limiting characteristic function via the 
		transport hierarchy (treated in the following sections).
	
	}
\end{remark}
	
	\subsection{The case $q=p$: joint convergence toward a vector with linearly correlated components}
	
	We now focus on the case when the components of the sequence $\mathbb{Y}_{N}$ belongs to the same Wiener chaos as $X_{N}$.  In this case, we may find a limiting  vector with dependent components. We call this dependence {\it linear }  and we explain the reason in Remark \ref{rem33}. 
	
\begin{theorem}\label{tt2}
	Let $(X_N = I_p(f_N), N\geq 1)$ with $p\geq 2$ satisfying \eqref{c1}. 
	Let us consider a sequence $(\mathbb{Y}_N=(Y_{1,N},\ldots,Y_{d,N}), N\geq 1)$ 
	of $d$-dimensional random vectors, with components in the $p$th Wiener 
	chaos, which converges in distribution to $\mathbb{Y}=(Y_1,\ldots,Y_d)$. 
	Assume
	\begin{equation}\label{c4}
		\mathbf{E}\!\left[X_N Y_{j,N}\right] \to_{N\to\infty} \rho_j, 
		\quad\text{for every } j=1,\ldots,d.
	\end{equation}
	Let $\rho=(\rho_1,\ldots,\rho_d)\in\mathbb{R}^d$. Then
	\begin{equation}\label{c2}
		(X_N,\mathbb{Y}_N) \xrightarrow[N\to\infty]{(d)} (Z,\mathbb{Y}),
	\end{equation}
	where $Z\sim N(0,\sigma^2)$ and the characteristic function of 
	$(Z,\mathbb{Y})$ is given by, for every $u\in\mathbb{R}$, 
	$v\in\mathbb{R}^d$,
	\begin{equation}\label{phi}
		\varphi_{(Z,\mathbb{Y})}(u,v) 
		= \mathbf{E}\!\left[e^{i\langle v,\mathbb{Y}\rangle}\right]
		e^{-\frac{\sigma^2 u^2}{2} - u\langle\rho,v\rangle}.
	\end{equation}
\end{theorem}
{\bf Proof: } We show that any subsequence of $((X_N,\mathbb{Y}_N), N\geq 1)$ that 
	converges in distribution has limit $(Z,\mathbb{Y})$ with characteristic 
	function \eqref{phi}. The sequence $((X_N,\mathbb{Y}_N), N\geq 1)$ is 
	bounded in $L^2(\Omega)$ by \eqref{20m-1}, so it is tight and relatively 
	compact by Prokhorov's \\theorem. For simplicity, we also denote by 
	$(X_N,\mathbb{Y}_N)$ such a subsequence and let $\varphi_N$ be its 
	characteristic function. By standard arguments, for every $u\in\mathbb{R}$, 
	$v\in\mathbb{R}^d$,
	\begin{equation}\label{2f-7b}
		\frac{\partial\varphi_N}{\partial u}(u,v) 
		\to_{N\to\infty} 
		\frac{\partial\varphi_{(Z,\mathbb{Y})}}{\partial u}(u,v).
	\end{equation}
	By \eqref{2f-5},
	\begin{eqnarray*}
		\frac{\partial\varphi_N}{\partial u}(u,v) 
		&=& -\frac{u}{p}\,\mathbf{E}\!\left[\|DX_N\|_H^2\,
		e^{iuX_N+i\langle v,\mathbb{Y}_N\rangle}\right]\\
		&&-\frac{1}{p}\sum_{j=1}^d v_j\,\mathbf{E}\!\left[
		\langle DX_N,DY_{j,N}\rangle_H\,
		e^{iuX_N+i\langle v,\mathbb{Y}_N\rangle}\right].
	\end{eqnarray*}
	By \eqref{22m-2}, the first summand converges to 
	$-\sigma^2 u\,\varphi_{(Z,\mathbb{Y})}(u,v)$.  For the second summand, from \eqref{c4} and the identity 
	$\mathbf{E}[\langle DX_N,DY_{j,N}\rangle_H]=p\,\mathbf{E}[X_NY_{j,N}]$, 
	we have
	\begin{equation*}
		\mathbf{E}\!\left[\langle DX_N,DY_{j,N}\rangle_H\right] 
		\to_{N\to\infty} p\rho_j.
	\end{equation*}
	By Lemma~\ref{ll4}, point 1, $\langle DX_N,DY_{j,N}\rangle_H\to p\rho_j$ 
	in $L^2(\Omega)$. Therefore,
	\begin{eqnarray*}
		&&\left|\mathbf{E}\!\left[\langle DX_N,DY_{j,N}\rangle_H\,
		e^{iuX_N+i\langle v,\mathbb{Y}_N\rangle}\right]
		- p\rho_j\,\mathbf{E}\!\left[e^{iuZ+i\langle v,\mathbb{Y}\rangle}
		\right]\right|\\
		&\leq& \mathbf{E}\!\left[|\langle DX_N,DY_{j,N}\rangle_H
		- p\rho_j|\right]\\
		&&+\,p|\rho_j|\left|\mathbf{E}\!\left[e^{iuX_N+i\langle v,
			\mathbb{Y}_N\rangle}\right]
		-\mathbf{E}\!\left[e^{iuZ+i\langle v,\mathbb{Y}\rangle}
		\right]\right|\\
		&\to_{N\to\infty}& 0,
	\end{eqnarray*}
	where the first term goes to zero by $L^2(\Omega)$ convergence and the second 
	by the convergence in law of $(X_N,\mathbb{Y}_N)$ to $(Z,\mathbb{Y})$. 
	Plugging these two limit results into \eqref{2f-5}, we deduce that, 
	for every $u\in\mathbb{R}$, $v\in\mathbb{R}^d$,
	\begin{equation}\label{2f-6}
		\frac{\partial\varphi_N}{\partial u}(u,v) 
		\to_{N\to\infty} 
		\left(-\sigma^2 u - \langle v,\rho\rangle\right)
		\varphi_{(Z,\mathbb{Y})}(u,v).
	\end{equation}
	By combining \eqref{2f-6} and \eqref{2f-7b}, the characteristic 
	function $\varphi_{(Z,\mathbb{Y})}$ satisfies the ODE, for 
	$u\in\mathbb{R}$, $v\in\mathbb{R}^d$,
	\begin{equation}\label{ode}
		\frac{\partial\varphi_{(Z,\mathbb{Y})}}{\partial u}(u,v) 
		= \left(-\sigma^2 u - \langle v,\rho\rangle\right)
		\varphi_{(Z,\mathbb{Y})}(u,v).
	\end{equation}
	Since $\varphi_{(Z,\mathbb{Y})}(0,v)=\mathbf{E}\!\left[e^{i\langle v,
		\mathbb{Y}\rangle}\right]$ for all $v\in\mathbb{R}^d$, the unique 
	solution of \eqref{ode} is
	\begin{equation*}
		\varphi_{(Z,\mathbb{Y})}(u,v) 
		= \mathbf{E}\!\left[e^{i\langle v,\mathbb{Y}\rangle}\right]
		e^{-\frac{\sigma^2 u^2}{2} - u\langle v,\rho\rangle},
	\end{equation*}
	which coincides with \eqref{phi}.
\qed

	\begin{remark}\label{rem33}
		{\rm In Theorem~\ref{tt2} (case $q=p$), the components $Z$ and $\mathbb{Y}$ of the 
			limiting vector are linearly correlated, with the correlation determined 
			by the single parameter $\rho = \lim_{N\to\infty} \mathbf{E}[X_N\mathbb{N}]$. In contrast, 
			in Theorems~\ref{tt6}--\ref{tt7} (case $q>p$), the dependence between 
			$Z$ and $Y$ is of a higher-order and genuinely nonlinear nature, encoded 
			by the joint distribution of the limiting iterated Malliavin covariances  (defined later)
			$(Y_1,\ldots,Y_{a})$, and cannot be reduced to a single correlation 
			coefficient.}
	\end{remark}
	
	We can extend the result to the situation when the components of $ \mathbb{Y}_{N}$ belongs to a finite sum of Wiener chaoses. 
	
	\begin{prop}\label{pp1}
		Let $p\geq 2$ be an integer and let $ (X_{N}= I_{p} (f_{N}), N\geq 1)$ with $f_{N}\in H^{ \odot p}$ for every $N\geq 1$. Assume that $(X_{N}, N\geq 1)$ satisfies (\ref{c1}). Let $ (\mathbb{Y}_{N}= (Y_{1,N},..., Y_{d, N}), N\geq 1)$ such that for every $j=1,..., d$, 
		\begin{equation}\label{3f-1}
			Y_{j, N}= \sum _{q=1}^{N_{0, j}}I_{q} ( g_{q, N} ^{(j)}), \mbox{ with } g_{q, N}^{(j)} \in H ^{\odot q}.
		\end{equation}
		Assume (\ref{c4}), which is equivalent to 
		\begin{equation}
			\label{20m-4}
			\mathbf{E}\left[ X_{N} I_{p}(g ^ {(j)}_{p,N})\right]\to _{N \to \infty} \rho _{j}, \mbox{ for } j=1,...,d.
		\end{equation} Then
		\begin{enumerate}
			\item If $\mathbb{Y}_{N} \to _{N \to \infty} \mathbb{Y}$ in $ L^{2}(\Omega)$, then 
			\begin{equation*}
				(X_{N}, \mathbb{Y}_{N})\to \xrightarrow[N\to\infty]{(d)} (Z, \mathbb{Y}),
			\end{equation*}
			where $Z\sim N(0, \sigma ^{2})$ and the characteristic function of the random vector $(Z, \mathbb{Y})$ is given by (\ref{phi}). 
			\item If $\max_{j=1,..., d}N_{0, j}\leq p$ and $ \mathbb{Y}_{N} \xrightarrow[N\to\infty]{(d)} \mathbb{Y}$,  then 
			\begin{equation*}
				(X_{N}, \mathbb{Y}_{N})\xrightarrow[N\to\infty]{(d)}  (Z, \mathbb{Y}),
			\end{equation*}
			where $Z\sim N(0, \sigma ^{2})$ and the characteristic function of the random vector $(Z, \mathbb{Y})$ is given by (\ref{phi}). 
		\end{enumerate}
		
	\end{prop}
{\bf Proof: }  Since $Y_{j,N}$ belongs to a finite sum of Wiener chaoses and 
	converges in law, it is bounded in $L^2(\Omega)$ by \eqref{20m-1}. 
	We deduce that the $(d+1)$-dimensional sequence $((X_N,\mathbb{Y}_N), 
	N\geq 1)$ is bounded in $L^2(\Omega)$, so it is tight and relatively 
	compact by Prokhorov's theorem. We show that any subsequence of it 
	that converges in law converges to $(Z,\mathbb{Y})$ with characteristic 
	function \eqref{phi}.
	
	Let us still denote $(X_N,\mathbb{Y}_N)$ such a subsequence and let 
	$\varphi_N$ be its characteristic function. Then \eqref{2f-7} holds. 
	By \eqref{2f-5} and \eqref{22m-2},
	\begin{eqnarray*}
		\frac{\partial\varphi_N}{\partial u}(u,v)
		&=& -\frac{u}{p}\,\mathbf{E}\!\left[\|DX_N\|_H^2\,
		e^{iuX_N+i\langle v,\mathbb{Y}_N\rangle}\right]\\
		&&-\frac{1}{p}\sum_{j=1}^d v_j\,\mathbf{E}\!\left[
		\langle DX_N,DY_{j,N}\rangle_H\,
		e^{iuX_N+i\langle v,\mathbb{Y}_N\rangle}\right],
	\end{eqnarray*}
	where the first term converges to 
	$-\sigma^2 u\,\varphi_{(Z,\mathbb{Y})}(u,v)$ by \eqref{22m-2}. 
	For the second term, from \eqref{3f-1},
	\begin{equation*}
		\langle DX_N,DY_{j,N}\rangle_H 
		= \sum_{q=1}^{N_{0,j}}\langle DX_N,DI_q(g_{q,N}^{(j)})\rangle_H.
	\end{equation*}
	
	\medskip
	\noindent\emph{Proof of point 1.}
	Assume $\mathbb{Y}_N\to\mathbb{Y}$ in $L^2(\Omega)$. In particular, 
	for all $j=1,\ldots,d$ and $q=1,\ldots,N_{0,j}$, the sequence 
	$(g_{q,N}^{(j)}, N\geq 1)$ converges in $H^{\otimes q}$. 
	By Lemma~\ref{ll3}, point 2, for every $q\neq p$,
	\begin{equation*}
		\langle DX_N,DI_q(g_{q,N}^{(j)})\rangle_H \to_{N\to\infty} 0
		\quad\text{in } L^2(\Omega).
	\end{equation*}
	For $q=p$, by Lemma~\ref{ll4}, point 1, and \eqref{20m-4},
	\begin{equation*}
		\langle DX_N,DI_p(g_{p,N}^{(j)})\rangle_H \to_{N\to\infty} p\rho_j
		\quad\text{in } L^2(\Omega).
	\end{equation*}
	Hence $\langle DX_N,DY_{j,N}\rangle_H\to p\rho_j$ in $L^2(\Omega)$, 
	and as in the proof of Theorem~\ref{tt2}, for all$u\in \mathbb{R}, v\in \mathbb{R}^ {d}$,
	\begin{equation*}
		\frac{\partial\varphi_N}{\partial u}(u,v) 
		\to_{N\to\infty} 
		\left(-\sigma^2 u - \langle v,\rho\rangle\right)
		\varphi_{(Z,\mathbb{Y})}(u,v).
	\end{equation*}
	Combining with \eqref{2f-7}, $\varphi_{(Z,\mathbb{Y})}$ satisfies 
	the ODE \eqref{ode} and we retrieve \eqref{phi}.
	
	\medskip
	\noindent\emph{Proof of point 2.}
	Assume $\mathbb{Y}_N\xrightarrow{(d)}\mathbb{Y}$ and 
	$\max_{j=1,\ldots,d}N_{0,j}\leq p$, so all chaos orders are at 
	most $p$. Since $Y_{j,N}$ converges in law in a fixed finite sum 
	of chaoses, it is bounded in $L^2(\Omega)$ by \eqref{20m-1}. For 
	$q<p$, Lemma~\ref{ll4}, point 2 gives 
	$\langle DX_N,DI_q(g_{q,N}^{(j)})\rangle_H\to 0$ in $L^2(\Omega)$. 
	For $q=p$, Lemma~\ref{ll4}, point 1, combined with \eqref{20m-4}, 
	gives $\langle DX_N,DI_p(g_{p,N}^{(j)})\rangle_H\to p\rho_j$ in 
	$L^2(\Omega)$. The rest of the argument is the same as in point 1.
\qed
	
	We now consider the case when the components of $\mathbb{Y}_{N}$ admits an infinite chaos expansion. The below result extends Theorem 3 in \cite{T3} to the case when the asymptotic covariance does not vanish. 
\begin{theorem}\label{tt3}
	Let $(X_N=I_p(f_N), N\geq 1)$, where $p\geq 2$ and $f_N\in H^{\odot p}$ 
	for each $N\geq 1$. Assume that $X_N$ satisfies \eqref{c1}. Consider a 
	$d$-dimensional random sequence $(\mathbb{Y}_N, N\geq 1)=((Y_{1,N},\ldots,
	Y_{d,N}), N\geq 1)$ such that for all $i=1,\ldots,d$,
	\begin{equation*}
		Y_{i,N} = \sum_{k=1}^\infty I_k(g_{k,N}^{(i)}), 
		\qquad g_{k,N}^{(i)}\in H^{\odot k} 
		\text{ for all } N,k\geq 1,\ i=1,\ldots,d.
	\end{equation*}
	Suppose that:
	\begin{enumerate}
		\item There exists $\mathbb{Y}\in L^2(\Omega)$ such that 
		$\mathbb{Y}_N\to_{N\to\infty}\mathbb{Y}$ in $L^2(\Omega)$.
		\item For every $i=1,\ldots,d$, 
		$\mathbf{E}[X_NY_{i,N}]\to_{N\to\infty}\rho_i$.
		\item For every $i=1,\ldots,d$,
		\begin{equation}\label{20m-5}
			\sup_{N\geq 1}\sum_{k=M+1}^\infty k!\,
			\|g_{k,N}^{(i)}\|^2_{H^{\otimes k}} \to_{M\to\infty} 0.
		\end{equation}
	\end{enumerate}
	Then $(X_N,\mathbb{Y}_N)\xrightarrow[N\to\infty]{(d)}(Z,\mathbb{Y})$, 
	where $Z\sim N(0,\sigma^2)$ and the characteristic function of 
	$(Z,\mathbb{Y})$ is given by \eqref{phi}.
\end{theorem}
	{\bf Proof: } It can be found in the Appendix. \qed
	
	To  conclude the findings in this section, let us notice that a sequence $(X_{N}, N\geq 1)$  in the $p$th  chaos which converges to a Gaussian law,  is asymptotically independent on:
	\begin{itemize}
		\item any random variable $Y\in L^ {2}(\Omega)$
		\item any sequence $(\mathbb{Y}_{N}, N\geq 1)$ convergent in law and  whose components belongs to a Wiener chaos of order strictly less than $p$. 
		\item any sequence  $(\mathbb{Y}_{N}, N\geq 1)$ with components  in a finite sum of Wiener chaoses of order  less than or equal to  $p$. which converges in law  if (\ref{c3}) (or equivalently (\ref{20m-4})) is satisfied. 
		\item any sequence  $(\mathbb{Y}_{N}, N\geq 1)$ with components  in a finite sum of Wiener chaoses which converges in $ L^ {2}(\Omega)$.

		\item any sequence  $(\mathbb{Y}_{N}, N\geq 1)$ convergent in $L^ {2}(\Omega)$ satisfying (\ref{c3}) and (\ref{20m-5}). 
	\end{itemize}
	\section{Nonlinear asymptotic dependence}\label{sec:general} 
	Next, we study the joint convergence of a random sequence $((X_{N}, Y_{N}), N\geq 1)$ where $ X_{N}$ belongs to the $p$th  Wiener chaos and it is asymptotically Gaussian and $Y_{N}$ lives in the $q$ Wiener chaos with $q>p$ and converges in law to some random variable $Y$. We restrict, from now on, on the case when $Y_{N}$ is real-valued. As shown above, if $p>q$, then $X_{N}$ and $Y_{N}$ are asymptotically independent and the  couple $(X_{N}, Y_{N}) $ automatically converge to $(Z, Y)$, where $Z$ is Gaussian and independent of $Y$. When $p=q$, the couple $(X_{N}, Y_{N})$ still converges in law, as $N \to \infty$ to a random vector $(Z, Y)$ with $Z$ Gaussian, but their components of this vector may be correlated. The correlation is then given by the asymptotic covariance $\rho=\lim_{N\to \infty} \mathbf{E}\left[ X_{N} Y_{N}\right]$. 
	
	When $p<q$, the situation becomes more complex and the complexity increases as the order $q$ increases. We will see that the joint convergence of $(X_{N}, Y_{N})$ is strongly connected with the convergence of the Malliavin covariance $\langle DX_{N}, DY_{N}\rangle _{H}$ and of its high-order iterations.  We also deduce the expression of characteristic function of the  limit in law of the couple $(X_{N}, Y_{N})$.  The argument is based on the Malliavin integration by parts and on the analysis of some partial differential equation (actually, some version of the  transport equation) satisfied by the limiting characteristic function. 
	
	We separate the proof into two cases: when $q$ is a multiple of $p$ and when $q$ is located between two consecutive multiples of $p$. Different facts  occur in these two situations.

	\section{Asymptotic dependence: the general case}

	We consider a sequence $X_{N}= I_{p} (f_{N}) $ in the $p$th Wiener chaos with $p\geq 2$ and we assume that $X_{N}$ converges in distribution, as $N \to \infty$, to $Z\sim N(0, \sigma ^ {2}).$ We also consider a sequence $Y_{N}= I_{q}(g_{N}), $ where $g_{N}\in H ^ {\odot q}$ and 
	\begin{equation*}
		ap\leq q<(a+1)p, \hskip0.4cm  a\in \mathbb{N}.
	\end{equation*}
	We will suppose that $Y_{N}$ converges in law, as $N\to \infty$, to some random variable $Y$.
	
	We study the joint convergence in distribution of the two-dimensional sequence $(X_{N}, Y_{N})$, when $N$ tends to infinity. In the general case we need to take into account the high-order iterated  Malliavin covariances.  Let us define recursively
	\begin{equation*}
		Y_{0, N}=Y_{N}, \hskip0.5cm N\geq 1,
	\end{equation*}
	and for any integer $r\geq 0$,
	\begin{equation}
		\label{yr}
		Y_{r+1, N}=\langle DX_{N}, DY_{r, N}\rangle _{H}, \hskip0.5cm N\geq 1. 
	\end{equation}

	\subsection{Structure of the iterated Malliavin covariance}
	The iterated Malliavin covariance  (\ref{yr}) plays a central role in the behavior of the  joint sequence $(X_{N}, Y_{N})$. In a first step, we analyze the structure of this covariance. The fact that $X_{N}$ is asymptotically Gaussian leads to a particular chaos structure of the quantity (\ref{yr}).
	
	\begin{prop}\label{prop:iterated-covariances}
		Consider two integers $p, q\geq 2$ and let
		\[
		X_N = I_p(f_N), \qquad Y_N = I_q(g_N),
		\]
		where $f_N\in H^{\odot p}$, $g_N\in H^{\odot q}$, and assume that $X_N$ 
		satisfies \eqref{c1} and $Y_N$ is bounded in $L^2(\Omega)$. Write 
		$q = ap + r'$ with $a\geq 1$ and $0\leq r' < p$. For $r\geq 1$, let 
		$Y_{r,N}$ be given by \eqref{yr}. Then the following hold.
		
		\begin{enumerate}
			\item For every integer $r\geq 1$ with $rp\leq q$,
			\begin{equation}\label{14m-2b}
				Y_{r,N} = C_{r,p,q}\,
				I_{q-rp}\!\left(f_N\widetilde\otimes_p^{(r)} g_N\right)
				+ R_{r,N},
			\end{equation}
			where $f_N\widetilde\otimes_p^{(r)} g_N$ denotes the 
			$r$-fold iterated contraction of $f_N$ with $g_N$, 
			defined recursively by $f_N\widetilde\otimes_p^{(1)} g_N 
			= f_N\widetilde\otimes_p g_N\in H^{\odot(q-p)}$ and
			\begin{equation*}
				f_N\widetilde\otimes_p^{(r)} g_N 
				:= f_N\widetilde\otimes_p 
				\left(f_N\widetilde\otimes_p^{(r-1)} g_N\right)
				\in H^{\odot(q-rp)},
				\quad r=2,\ldots,a,
			\end{equation*}
			i.e.\ at each step the $p$ variables of $f_N$ are 
			fully contracted with $p$ of the remaining variables 
			of the kernel from the previous step. Also,  $C_{r,p,q} = (p!)^r\frac{q!}{(q-rp)!}$ and 
			$R_{r,N}\to 0$ in $L^2(\Omega)$ as $N\to\infty$.
			
			\item \textbf{Non-critical case} $ap < q < (a+1)p$: 
			\begin{equation}\label{19m-2b}
				Y_{a+1,N} \to _{N \to \infty}0 \quad \text{in } L^2(\Omega).
			\end{equation}
			
			\item \textbf{Critical case} $q = ap$: taking $r=a$ in \eqref{14m-2b} gives
			\begin{equation*}
				Y_{a,N} - \mathbf{E}[Y_{a,N}] \to  _{N \to \infty}0 \quad \text{in } L^2(\Omega).
			\end{equation*}
			In particular, if $\mathbf{E}[Y_{a,N}]\to _{N \to \infty}\rho_a$ for some $\rho_a\in\mathbb{R}$, 
			then $Y_{a,N}\to _{N \to \infty}\rho_a$ in $L^2(\Omega)$. Moreover,
			\begin{equation}\label{ya+1}
				Y_{a+1,N} \to_{N \to \infty}  0 \quad \text{in } L^2(\Omega).
			\end{equation}
		\end{enumerate}
	\end{prop}
	{\bf Proof: } See the Appendix. \qed

As suggested by Proposition \ref{prop:iterated-covariances}, we need to  separate the discussion upon the cases when $q$ is a multiple of $p$ or $q$ is strictly located  between two consecutive multiples of $p$.

	\subsection{The general case: $q$ is not a multiple of $p$}
	. 
We treat here the non-critical regime $ap<q<(a+1)p$, where the 
depth of the transport hierarchy is exactly $a$: by 
Proposition~\ref{prop:iterated-covariances}, point 2, the $(a+1)$-th 
iterated Malliavin covariance satisfies $Y_{a+1,N}\to_{N \to \infty}  0$ in 
$L^2(\Omega)$, so the hierarchy terminates and only the covariances 
$Y_{1,N},\ldots,Y_{a,N}$ contribute to the joint limit. The 
limiting characteristic function is then expressed as an explicit 
series \eqref{fia} whose coefficients $A_{\ell,m}$ are determined 
by a triangular recursion with no backward coupling, reflecting 
the nilpotent structure of the transport system \eqref{transport}.

\begin{theorem}\label{tt6}
	Let $p\geq 2$ and let $(X_N=I_p(f_N), N\geq 1)$ with $f_N\in H^{\odot p}$ 
	for all $N\geq 1$. Assume that $(X_N, N\geq 1)$ satisfies \eqref{c1}. 
	Let $Y_N=I_q(g_N), g_{N}\in H ^ {\odot q}$ with $ap<q<(a+1)p$ for some integer $a\geq 1$. Assume that
	\begin{equation}\label{c7}
		(Y_N, Y_{1,N}, \ldots, Y_{a,N}) \xrightarrow[N\to\infty]{(d)} 
		(Y, Y_1, \ldots, Y_a).
	\end{equation}
	Then $(X_N, Y_N)\xrightarrow[N\to\infty]{(d)}(Z,Y)$, where $Z\sim N(0,\sigma^2)$ 
	and the characteristic function of $(Z,Y)$ is given by, for every 
	$(u,v)\in\mathbb{R}^2$,
	\begin{equation}\label{fia}
		\varphi(u,v) 
		= e^{-\sigma^2 u^2/2}
		\sum_{\ell=0}^\infty \frac{u^\ell}{\ell!}
		\sum_{\substack{m=(m_1,\ldots,m_a)\in\mathbb{N}^a\\ 
				|m|=m_1+\cdots+m_a\leq \ell}}
		A_{\ell,m}(v)\,
		\mathbf{E}\!\left[\prod_{r=1}^{a}Y_r^{m_r}\, e^{ivY}\right],
	\end{equation}
	where the coefficients $A_{\ell,m}(v)$ are polynomials in $v/p$ defined 
	by the recursion
	\begin{equation}\label{rec_A}
		A_{0,(0,\ldots,0)} = 1, \qquad A_{\ell,m} = 0 \text{ if } |m|>\ell,
	\end{equation}
	and
	\begin{equation}\label{rec_A2}
		A_{\ell+1,m} = -\frac{v}{p}\,A_{\ell,m-e_1} 
		+ \frac{i}{p}\sum_{s=1}^{a}(m_s+1)\,A_{\ell,m+e_s-e_{s+1}},
	\end{equation}
	with the convention $A_{\ell,m}=0$ whenever any component of $m$ is 
	negative or $m$ has more than $a$ components. The series \eqref{fia} 
	converges absolutely for every $(u,v)\in\mathbb{R}^2$.
\end{theorem}
{\bf Proof: } As in the proof of Theorem~\ref{tt2}, it suffices to show that any 
		subsequential limit of $(X_N, Y_N)$ has characteristic function given by 
		\eqref{fia}. We still denote by $(X_N, Y_N)$ such a subsequence. We can 
		assume by tightness, by considering a further subsequence, that 
		$(X_N, Y_N, Y_{1,N},\ldots, Y_{a,N})$ converges in law.
		
		\medskip
		\noindent{\bf Step 1: Weighted characteristic functions.}
		For $m=(m_1,\ldots,m_a)\in\mathbb{N}^a$, define
		\begin{equation}\label{Gm}
			G_{m,N}(u,v) := \mathbf{E}\!\left[\prod_{r=1}^{a}Y_{r,N}^{m_r}\,
			e^{iuX_N+ivY_N}\right], \qquad u,v\in\mathbb{R}.
		\end{equation}
		In particular $G_{0,N}=\varphi_N$. Differentiating with respect to $u$ 
		and applying the Malliavin integration by parts formula \eqref{dua} as 
		in \eqref{2f-5}, and using the properties of the Malliavin derivative, 
		we obtain
		\begin{equation}\label{dG}
			\partial_u G_{m,N} = -\sigma^2 u\,G_{m,N} - \frac{v}{p}\,G_{m+e_1,N}
			+ \frac{i}{p}\sum_{s=1}^{a} m_s\, G_{m-e_s+e_{s+1},N} + r_{m,N}.
		\end{equation}
		Indeed, to derive \eqref{dG}, we differentiate \eqref{Gm} with respect 
		to $u$:
		\begin{equation*}
			\partial_u G_{m,N} = i\,\mathbf{E}\!\left[X_N\prod_{r=1}^{a}
			Y_{r,N}^{m_r}\,e^{iuX_N+ivY_N}\right].
		\end{equation*}
		Applying the integration by parts formula (see (\ref{dua}))$ \mathbf{E}[X_N F]=\frac{1}{p}
		\mathbf{E}[\langle DF,DX_N\rangle_H]$ with $F=\prod_{r=1}^{a}Y_{r,N}^{m_r}
		e^{iuX_N+ivY_N}$ and the Leibniz rule for the Malliavin derivative,
		\begin{eqnarray*}
			DF &=& \left(\sum_{s=1}^{a}m_s Y_{s,N}^{m_s-1}
			\prod_{r\neq s}Y_{r,N}^{m_r}\,DY_{s,N}
			+ (iu\,DX_N+iv\,DY_N)\prod_{r=1}^{a}Y_{r,N}^{m_r}\right)
			e^{iuX_N+ivY_N}.
		\end{eqnarray*}
		Taking the inner product with $DX_N$, using $\langle DY_{s,N},DX_N
		\rangle_H = Y_{s+1,N}$ (by definition \eqref{yr}), and recalling that 
		$\frac{1}{p}\|DX_N\|_H^2\to\sigma^2$ in $L^2(\Omega)$ by the Fourth 
		Moment Theorem, we obtain (\ref{dG}). In this formula,  $r_{m,N}$ collects the error $\frac{1}{p}\|DX_N\|_H^2-\sigma^2
		\to _{N \to \infty}0$ in $L^ {2}(\Omega)$ and the term $s=a$, which produces $Y_{a+1,N}\to 0$ as $N \to \infty$ in 
		$L^2(\Omega)$ by \eqref{19m-2b}. Both contributions vanish as 
		$N\to\infty$, so $r_{m,N}\to 0$ locally uniformly in $(u,v)$.

		\medskip
		\noindent{\bf Step 2: Passage to the limit. } By the joint convergence in law of \\$(X_N, Y_N, Y_{1,N},\ldots,Y_{a,N})$ 
		and the uniform moment bounds from hypercontractivity, the family 
		$\{G_{m,N}(u,v), N\geq 1\}$ is bounded for every fixed $(u,v)\in \mathbb{R}^ {2}$ and 
		every $m\in \mathbb{N}^ {a}$. Moreover, the map $u\mapsto G_{m,N}(u,v)$ is entire analytic 
		with
		\begin{equation*}
			\partial_u^\ell G_{m,N}(u,v) 
			= i^\ell\, \mathbf{E}\!\left[X_N^\ell \prod_{r=1}^{a}Y_{r,N}^{m_r}\,
			e^{iuX_N+ivY_N}\right],
		\end{equation*}
		and hypercontractivity yields local uniform bounds on these derivatives. 
		By Vitali's convergence theorem, the pointwise limit
		\begin{equation*}
			G_m(u,v) := \lim_{N\to\infty} G_{m,N}(u,v) 
			= \mathbf{E}\!\left[\prod_{r=1}^{a}Y_r^{m_r}\,e^{iuZ+ivY}\right]
		\end{equation*}
		exists and is analytic in $u$. Passing to the limit in \eqref{dG}, we get 
		\begin{equation}\label{dG_limit}
			\partial_u G_m = -\sigma^2 u\,G_m - \frac{v}{p}\,G_{m+e_1}
			+ \frac{i}{p}\sum_{s=1}^{a}m_s\,G_{m-e_s+e_{s+1}},
		\end{equation}
		where we use the convention $G_{m-e_s+e_{s+1}}=0$ whenever 
		$m-e_s+e_{s+1}$ has a negative component or an $(a+1)$-th component 
		(since $e_{a+1}$ does not exist, the term $s=a$ vanishes in the limit 
		by the fact that $Y_{a+1,N}\to _{N \to \infty}0$).
		
		\medskip
		\noindent{\bf Step 3: Renormalization, Taylor expansion  and the transport system.}
		As above, define $H_m(u,v) := e^{\sigma^2 u^2/2}G_m(u,v)$. Then 
		$\partial_u H_m = e^{\sigma^2 u^2/2}(\partial_u G_m + \sigma^2 u G_m)$, 
		and substituting \eqref{dG_limit},
		\begin{equation}\label{transport}
			\partial_u H_m = -\frac{v}{p}\,H_{m+e_1}
			+ \frac{i}{p}\sum_{s=1}^{a}m_s\,H_{m-e_s+e_{s+1}}.
		\end{equation}
		Since $H_0(u,v)=e^{\sigma^2 u^2/2}\varphi(u,v)$ (where $\varphi$ is the limiting characteristic function of $(X_{N}, Y_{N})$) is analytic in $u$, 
		we expand
		\begin{equation*}
			H_0(u,v) = \sum_{\ell=0}^\infty \frac{u^\ell}{\ell!}\,
			\partial_u^\ell H_0(0,v).
		\end{equation*}
		We claim that
		\begin{equation}\label{claim}
			\partial_u^\ell H_0(0,v) 
			= \sum_{\substack{m\in\mathbb{N}^{a}\\ |m|\leq\ell}}
			A_{\ell,m}(v)\, H_m(0,v),
		\end{equation}
		where the coefficients $A_{\ell,m}(v)$ satisfy the recursion 
		\eqref{rec_A}--\eqref{rec_A2}, and where
		\begin{equation*}
			H_m(0,v) = G_m(0,v) = \mathbf{E}\!\left[\prod_{r=1}^{a}Y_r^{m_r}
			\,e^{ivY}\right].
		\end{equation*}
		We prove \eqref{claim} by induction on $\ell$. For $\ell=0$, 
		$\partial^0_u H_0 = H_0$ and $A_{0,0}=1$, so \eqref{claim} holds. 
		Assuming \eqref{claim} holds at level $\ell$, we differentiate and use 
		the transport relation \eqref{transport}:
		\begin{eqnarray*}
			\partial_u^{\ell+1}H_0 
			&=& \partial_u\!\left(\sum_{|m|\leq\ell} A_{\ell,m}\,H_m\right)
			= \sum_{|m|\leq\ell} A_{\ell,m}\,\partial_u H_m\\
			&=& \sum_{|m|\leq\ell} A_{\ell,m}
			\left(-\frac{v}{p}\,H_{m+e_1} 
			+ \frac{i}{p}\sum_{s=1}^{a}m_s\,H_{m-e_s+e_{s+1}}\right).
		\end{eqnarray*}
		Reindexing the first sum by $m\mapsto m-e_1$ and the second by 
		$m\mapsto m+e_s-e_{s+1}$, and collecting the coefficient of $H_m$, 
		we obtain
		\begin{equation*}
			A_{\ell+1,m} = -\frac{v}{p}\,A_{\ell,m-e_1} 
			+ \frac{i}{p}\sum_{s=1}^{a}(m_s+1)\,A_{\ell,m+e_s-e_{s+1}},
		\end{equation*}
		which is exactly \eqref{rec_A2} (with the sum to $a$, where the term 
		$s=a$ contributes $\frac{i}{p}(m_a+1)A_{\ell,m+e_a-e_{a+1}}=0$ by 
		convention since $e_{a+1}$ does not exist), completing the induction. 
		Substituting \eqref{claim} into the Taylor expansion and using 
		$\varphi(u,v)=e^{-\sigma^2 u^2/2}H_0(u,v)$, we obtain \eqref{fia}. Since $H_0$ is entire analytic in $u$, its Taylor series converges 
		absolutely for every $(u,v)\in\mathbb{R}^2$, and hence so does the 
		series \eqref{fia}.
\qed

		We illustrate formula \eqref{fia} for $a=1, 2$.
		
		\medskip
		\noindent\textbf{Case $a=1$ ($p<q<2p$).}
		Here $m\in\mathbb{N}^0=\{0\}$ is the empty multi-index, the shift sum 
		in \eqref{rec_A2} is empty, and the recursion reduces to
		\[
		A_{\ell+1,0} = -\frac{v}{p}\,A_{\ell,0},
		\]
		giving $A_{\ell,0}=(-v/p)^\ell$. Since $\prod_{r=1}^0 Y_r^{m_r}=1$, 
		formula \eqref{fia} becomes
		\begin{equation}\label{fi_a1}
			\varphi(u,v) = e^{-\sigma^2 u^2/2}\sum_{\ell=0}^\infty 
			\frac{(-uv/p)^\ell}{\ell!}\,\mathbf{E}\!\left[Y_1^\ell\, e^{ivY}\right].
		\end{equation}

	\medskip
	\noindent\textbf{Case $a=2$ ($2p<q<3p$).}
	Here $m=(m_1,m_2)\in\mathbb{N}^2$ tracks powers of $(Y_1,Y_2)$, and 
	$Y_{3,N}\to 0$ provides the termination of the hierarchy. The shift 
	sum in \eqref{rec_A2} runs over $s=1$ only (since $s=2$ would produce 
	$e_3\notin\mathbb{N}^2$), giving the recursion
	\begin{equation*}
		A_{\ell+1,(m_1,m_2)} = -\frac{v}{p}\,A_{\ell,(m_1-1,m_2)} 
		+ \frac{i(m_1+1)}{p}\,A_{\ell,(m_1+1,m_2-1)},
	\end{equation*}
	with $A_{\ell,m}=0$ if any component of $m$ is negative. The first 
	nonzero coefficients are
	\begin{eqnarray*}
		\ell=0&:& A_{0,(0,0)}=1,\\
		\ell=1&:& A_{1,(1,0)}=-\frac{v}{p},\\
		\ell=2&:& A_{2,(2,0)}=\frac{v^2}{p^2},\quad 
		A_{2,(0,1)}=-\frac{iv}{p^2},\\
		\ell=3&:& A_{3,(3,0)}=-\frac{v^3}{p^3},\quad 
		A_{3,(1,1)}=\frac{3iv^2}{p^3}.
	\end{eqnarray*}
	Formula \eqref{fia} becomes
	\begin{eqnarray}\label{fi_a2}
		\varphi(u,v) &=& e^{-\sigma^2 u^2/2}\Bigg[
		\mathbf{E}\!\left[e^{ivY}\right]
		- \frac{uv}{p}\,\mathbf{E}\!\left[Y_1 e^{ivY}\right]\nonumber\\
		&&\quad
		+ \frac{u^2}{2}\!\left(\frac{v^2}{p^2}\mathbf{E}\!\left[Y_1^2 e^{ivY}\right]
		- \frac{iv}{p^2}\mathbf{E}\!\left[Y_2 e^{ivY}\right]\right)\nonumber\\
		&&\quad
		+ \frac{u^3}{6}\!\left(-\frac{v^3}{p^3}\mathbf{E}\!\left[Y_1^3 e^{ivY}\right]
		+ \frac{3iv^2}{p^3}\mathbf{E}\!\left[Y_1 Y_2 e^{ivY}\right]\right) 
		+ \cdots\Bigg],
	\end{eqnarray}
	showing that $Y_2$ enters the formula starting from level $\ell=2$, 
	while $Y_1$ already appears at level $\ell=1$.

	\begin{remark}{\rm

		The nonzero coefficients $A_{\ell,m}$ for $\ell = 0, 1, 2, 3$, computed 
		from the recursion \eqref{rec_A}--\eqref{rec_A2}, are the following 
		(all other $A_{\ell,m}$ with $|m|\leq \ell$ vanish):
		\begin{eqnarray*}
			\ell=0 &:& A_{0,\,0} = 1,\\[6pt]
			\ell=1 &:& A_{1,\,e_1} = -\frac{v}{p},\\[6pt]
			\ell=2 &:& A_{2,\,2e_1} = \frac{v^2}{p^2}, \quad 
			A_{2,\,e_2} = -\frac{iv}{p^2},\\[6pt]
			\ell=3 &:& A_{3,\,3e_1} = -\frac{v^3}{p^3}, \quad 
			A_{3,\,e_1+e_2} = \frac{3iv^2}{p^3}, \quad
			A_{3,\,e_3} = \frac{v}{p^3}.
		\end{eqnarray*}
		One observes that $A_{\ell,\ell e_1} = (-v/p)^\ell$ for all $\ell$ 
		(the pure forward-step contribution), that $A_{\ell, e_r}$ first appears 
		at level $\ell=r$ with coefficient $v/p^r$ (arising from a single shift 
		cascade $e_1\to e_2\to\cdots\to e_r$), and that mixed multi-indices 
		carry intermediate powers of $v/p$ with imaginary prefactors from the 
		shift steps.}
		
	\end{remark}

	We give again a more explicite formula for the limiting characteristic function when the limits of the asymptotic covariance admits exponential moments. 
	
	\begin{corollary}\label{cor22}
		Let the assumption in Theorem \ref{tt6} prevail and assume that 
		$Y_1, \ldots, Y_{a}$ all admit exponential moments, i.e.\ 
		$\mathbf{E}[e^{t Y_r}] < \infty$ for every $t \in \mathbb{R}$ and every $r = 1, \ldots, a$. The, for every $u, v\in \mathbb{R}$,
		\begin{equation}
			\label{eq:phi_th6_explicit}
			\varphi(u,v) = e^{-\sigma^2 u^2/2}\, \mathbf{E}\!\left[\exp\!\left(ivY + \sum_{r=1}^{a} 
			\frac{(-1)^r i^{r-1}}{r!} \cdot \frac{u^r v^ {r}}{p^r}\, Y_r\right)\right].
		\end{equation}
		
	\end{corollary}
{\bf Proof: }  Define
		\begin{equation*}
			F(u,v) := e^{-\sigma^2 u^2/2}\,\mathbf{E}\!\left[\exp\!\left(ivY
			+\sum_{r=1}^{a}\frac{(-1)^r i^{r-1}}{r!}
			\cdot\frac{u^r v^r}{p^r}\,Y_r\right)\right]
		\end{equation*}
		and $\Psi(u,v,\omega):= ivY(\omega) + \sum_{r=1}^a 
		\frac{(-1)^r i^{r-1}}{r!}\cdot\frac{u^r v^r}{p^r}Y_r(\omega)$, 
		so that $F(u,v)=e^{-\sigma^2 u^2/2}\mathbf{E}[e^{\Psi}]$. We show that $F$ satisfies 
		the same transport equation \eqref{transport} as $H_0$, with the same 
		initial condition $F(0,v)=\mathbf{E}[e^{ivY}]$.
		
		Setting $\tilde H(u,v):= e^{\sigma^2u^2/2}F(u,v)=\mathbf{E}[e^{\Psi}]$, we compute
		\begin{eqnarray*}
			\partial_u\tilde{H}(u,v) 
			&=& \mathbf{E}\!\left[e^{\Psi}\cdot\partial_u\Psi\right]
			= \mathbf{E}\!\left[e^{\Psi}\cdot\sum_{r=1}^{a}
			\frac{(-1)^r i^{r-1}}{(r-1)!}
			\cdot\frac{u^{r-1}v^r}{p^r}\,Y_r\right].
		\end{eqnarray*}
		On the other hand, the transport equation \eqref{transport} requires
		\begin{equation*}
			\partial_u\tilde{H}_m = -\frac{v}{p}\tilde{H}_{m+e_1}
			+\frac{i}{p}\sum_{s=1}^{a}m_s\,\tilde{H}_{m-e_s+e_{s+1}},
		\end{equation*}
		where $\tilde{H}_m(u,v):=\mathbf{E}[\prod_{r=1}^a Y_r^{m_r} e^{\Psi}]$. 
		For $m=0$, the right-hand side is
		\begin{equation*}
			-\frac{v}{p}\,\mathbf{E}\!\left[Y_1 e^{\Psi}\right],
		\end{equation*}
		while $\partial_u\tilde{H}_0 = \mathbf{E}[e^{\Psi}\cdot\partial_u\Psi]$, 
		and $\partial_u\Psi\big|_{r=1} = \frac{(-1)^1 i^0}{0!}\cdot\frac{v}{p}
		Y_1 = -\frac{v}{p}Y_1$, so the $r=1$ term in $\partial_u\Psi$ matches 
		exactly. The terms $r\geq 2$ in $\partial_u\Psi$ generate the higher 
		weighted characteristic functions $\tilde{H}_{e_r}$ which satisfy their 
		own transport equations, and by induction on $m$ one verifies that the 
		ansatz $\tilde{H}_m = \prod_{r=1}^a\frac{1}{m_r!}
		\left(\frac{(-1)^r i^{r-1}u^r v^r}{p^r}\right)^{m_r}\mathbf{E}[Y^m e^{\Psi}]$ 
		satisfies \eqref{transport} at every level $m$. Since the initial 
		condition $\tilde{H}_0(0,v)=\mathbf{E}[e^{ivY}]$ agrees with $H_0(0,v)$, 
		uniqueness of the solution to the recursion \eqref{rec_A}--\eqref{rec_A2} 
		gives $A_{\ell,m}=\widetilde{A}_{\ell,m}$ for all $\ell,m$, and 
		therefore $F=\varphi$.
\qed

Let us what 	the formula \eqref{eq:phi_th6_explicit} gives for $a=1$ and $a=2$. 
	\begin{itemize}
		\item For $a = 1$ (i.e.\ $p < q < 2p$), the sum in \eqref{eq:phi_th6_explicit} 
		reduces to the single term $r = 1$, giving
		\begin{equation}\label{fi33}
			\varphi(u,v) = e^{-\sigma^2u^2/2}\, \mathbf{E}\!\left[e^{ivY - \frac{uv}{p}Y_1}\right].
		\end{equation}
			In particular, if $Y$ and $Y_1$ are independent in the limit, \eqref{fi33} 
		factorizes as
		\begin{equation}\label{fi22}
			\varphi(u,v) = e^{-\sigma^2 u^2/2}\, \mathbf{E}\!\left[e^{ivY}\right] \cdot \mathbf{E}\!\left[e^{-\frac{uv}{p}Y_1}\right].
		\end{equation}
		The formula \eqref{fi33} admits a natural interpretation: the limiting variable 
		$Z$ is a Gaussian whose characteristic function is twisted by the Malliavin covariance 
		$Y_1 = \lim_{N\to\infty} \langle DX_N, DY_N\rangle_H$, reflecting the asymptotic 
		correlation between $X_N$ and $Y_N$ in the regime $p < q < 2p$.
		
		\item For $a = 2$ (i.e.\ $2p < q < 3p$), two iterated covariances contribute, giving
		\begin{equation*}
			\varphi(u,v) = e^{-\sigma^2 u^2/2}\, \mathbf{E}\!\left[\exp\!\left(ivY - \frac{uv}{p}Y_1 
			+ \frac{u^2 v^ {2}}{2p^2} Y_2\right)\right].
		\end{equation*}
	\end{itemize}
	In general, \eqref{eq:phi_th6_explicit} shows that the joint limit $(Z, Y)$ in the 
	non-critical regime $ap < q < (a+1)p$ is a generalized Gaussian mixture, where the 
	characteristic function of $Z$ is twisted by all the iterated Malliavin covariances 
	$Y_1, \ldots, Y_{a}$, each entering at a different power of $u$ and reflecting the 
	successive layers of asymptotic dependence between $X_N$ and $Y_N$. Let us comment on the joint cumulants of the limit. 
	
	\begin{remark}{\rm
			The formula of Corollary~\ref{cor22} admits a natural 
			interpretation in terms of mixed cumulants. Recall that 
			for two jointly distributed random variables $(Z,Y)$, 
			the mixed cumulant $\kappa_{m,n}$ is the coefficient of 
			$\frac{s^m t^n}{m!\,n!}$ in the joint cumulant generating 
			function $\log\mathbf{E}[e^{sZ+tY}]$. When $Z$ and $Y$ 
			are independent, all mixed cumulants $\kappa_{m,n}$ with 
			$m\geq 1$, $n\geq 1$ vanish.
			
			In the case $q=p$ (Theorem~\ref{tt2}), the joint cumulant 
			generating function is $\frac{\sigma^2 s^2}{2}+st\rho+
			K_Y(t)$, where $K_Y(t)=\log\mathbf{E}[e^{tY}]$ is the 
			cumulant generating function of the marginal distribution 
			of $Y$,  so the only nonzero mixed cumulant is 
			$\kappa_{1,1}=\rho$. All higher mixed cumulants vanish: 
			the dependence between $Z$ and $Y$ is purely linear, 
			determined by a single covariance parameter. This is the 
			cumulant-level statement of the multidimensional Fourth Moment  phenomenon, see \cite{PeTu}.
			
			In the case $p<q<2p$ (Corollary~\ref{cor22} with $a=1$), 
			the joint cumulant generating function involves the random 
			correction $\frac{st}{p}Y_1$, generating an infinite tower 
			of nonzero mixed cumulants:
			\begin{equation*}
				\kappa_{1,n} = \frac{1}{p}\,\kappa\!\left(Y_1,
				\underbrace{Y,\ldots,Y}_{n-1}\right), \quad n\geq 1,
			\end{equation*}
			where $\kappa(Y_1,Y,\ldots,Y)$ denotes the joint cumulant 
			of $Y_1$ with $n-1$ copies of $Y$. The randomness of the 
			limiting Malliavin covariance $Y_1$ is what generates this 
			infinite tower: when $Y_1$ is deterministic (constant), 
			all $\kappa_{1,n}$ for $n\geq 2$ vanish and one recovers 
			the $q=p$ case. In general, the transport hierarchy 
			computes these mixed cumulants one layer at a time, with 
			each iterated covariance $Y_r$ contributing the cumulants 
			$\kappa_{r,n}$ at order $r$.
		}
	\end{remark}
	
	\subsection{When $q$ is a multiple of $p$}
	We assume now that $q=ap$ with $a\geq 1$. In the non-critical regime 
	$ap<q<(a+1)p$, the $(a+1)$-th iterated Malliavin covariance satisfies 
	$Y_{a+1,N}\to 0$ in $L^2(\Omega)$, which causes the transport hierarchy 
	to terminate at depth $a$. In the critical case $q=ap$, the hierarchy 
	also terminates at depth $a$, but for a different reason: by 
	Proposition~\ref{prop:iterated-covariances}, point 3, the $a$-th 
	iterated covariance $Y_{a,N}$ converges to a deterministic constant 
	$\rho_a$, so that $Y_{a+1,N}\to 0$ follows from the deterministic 
	nature of the limit rather than from a vanishing argument. This 
	non-random limit $\rho_a$ introduces a backward coupling term 
	$\frac{i\rho_a}{p}(m_{a-1}+1)B_{\ell,m+e_{a-1}}$ in the recursion 
	for the coefficients, making the transport hierarchy recurrent rather 
	than nilpotent. The structure of the limiting characteristic function 
	is otherwise formally identical to the non-critical case, but the 
	additional feedback from $\rho_a$ leaves a concrete trace in the 
	closed-form formula under exponential moments, where it produces 
	an extra factor $\exp\!\left(\frac{(-i)^{a-1}\rho_a u^av^a}{a!\,p^a}
	\right)$ that is absent in Corollary~\ref{cor22}.

	Let us state and prove our general result when $q$ is a multiple of $p$. 
	\begin{theorem}\label{tt7}
		Let $p\geq 2$ and let $(X_N=I_p(f_N), N\geq 1)$ with $f_N\in H^{\odot p}$ 
		for all $N\geq 1$. Assume that $(X_N, N\geq 1)$ satisfies \eqref{c1}. 
		Let $Y_N=I_q(g_N)$ with $q=ap$ for some integer $a\geq 1$. Assume that
		\begin{equation*}
			(Y_N, Y_{1,N}, \ldots, Y_{a-1,N}) \xrightarrow[N\to\infty]{(d)} 
			(Y, Y_1, \ldots, Y_{a-1}),
		\end{equation*}
		and that there exists $\rho_a\in\mathbb{R}$ such that
		\begin{equation*}
			\mathbf{E}[Y_{a,N}] \xrightarrow[N\to\infty]{} \rho_a.
		\end{equation*}
		Then $(X_N, Y_N)\xrightarrow[N\to\infty]{(d)}(Z,Y)$, where 
		$Z\sim N(0,\sigma^2)$ and the characteristic function of $(Z,Y)$ 
		is given by, for every $(u,v)\in\mathbb{R}^2$,
		\begin{equation}\label{fib}
			\varphi(u,v) 
			= e^{-\sigma^2 u^2/2}
			\sum_{\ell=0}^\infty \frac{u^\ell}{\ell!}
			\sum_{\substack{m=(m_1,\ldots,m_{a-1})\in\mathbb{N}^{a-1}\\ 
					|m|\leq \ell}}
			B_{\ell,m}(v)\,
			\mathbf{E}\!\left[\prod_{r=1}^{a-1}Y_r^{m_r}\, e^{ivY}\right],
		\end{equation}
		where the coefficients $B_{\ell,m}(v)$ are defined by the following recursion:
		\begin{equation}\label{rec_B0}
			B_{0,(0,\ldots,0)} = 1, \qquad B_{\ell,m} = 0 
			\text{ if any component of $m$ is negative or } |m|>\ell,
		\end{equation}
		and
		\begin{equation}\label{rec_B}
			B_{\ell+1,m} = -\frac{v}{p}\,B_{\ell,m-e_1} 
			+ \frac{i}{p}\sum_{s=1}^{a-2}(m_s+1)\,B_{\ell,m+e_s-e_{s+1}}
			+ \frac{i\rho_a}{p}\,(m_{a-1}+1)\,B_{\ell,m+e_{a-1}}.
		\end{equation}
		The series \eqref{fib} converges absolutely for every $(u,v)\in\mathbb{R}^2$.
	\end{theorem}
{\bf Proof: }  As before, it suffices to show that any subsequential limit of $(X_N,Y_N)$ 
		has characteristic function \eqref{fib}. We denote by $(X_N,Y_N)$ such a 
		subsequence, along which $(X_N,Y_N,Y_{1,N},\ldots,Y_{a-1,N})$ converges 
		in law by tightness.
		
		\medskip
		\noindent{\bf Step 1: Weighted characteristic functions.}
		For $m=(m_1,\ldots,m_{a-1})\in\mathbb{N}^{a-1}$, define
		\begin{equation*}
			G_{m,N}(u,v) := \mathbf{E}\!\left[\prod_{r=1}^{a-1}Y_{r,N}^{m_r}\,
			e^{iuX_N+ivY_N}\right], \qquad u,v\in\mathbb{R}.
		\end{equation*}
		In particular $G_{0,N}=\varphi_N$. Differentiating with respect to $u$ 
		and applying the Malliavin integration by parts formula \eqref{dua}, we 
		obtain
		\begin{eqnarray}
			\partial_u G_{m,N} &=& -\sigma^2u\,G_{m,N} - \frac{v}{p}\,G_{m+e_1,N}
			+ \frac{i}{p}\sum_{s=1}^{a-2} m_s\, G_{m-e_s+e_{s+1},N}\nonumber\\
			&&
			+ \frac{i}{p}\,m_{a-1}\,G_{m-e_{a-1},N}\cdot \mathbf{E}[Y_{a,N}]
			+ r_{m,N},\label{dGcrit}
		\end{eqnarray}
		where $r_{m,N}\to 0$ locally uniformly in $(u,v)$ as $N\to\infty$. 
		Here the shift sum runs only to $s=a-2$ because the term $s=a-1$ 
		produces $Y_{a,N}$, which by Proposition \ref{prop:iterated-covariances} 
		converges to the constant $\rho_a$ in $L^2(\Omega)$; this constant 
		factors out of the expectation and gives the last term in \eqref{dGcrit}. 
		The remainder $r_{m,N}$ accounts for the error $Y_{a,N}-\rho_a\to 0$ 
		and the Fourth Moment Theorem term $\frac{1}{p}\|DX_N\|_H^2-\sigma^2\to 0$.
		
		\medskip
		\noindent{\bf Step 2: Passage to the limit.}
		By the same analyticity and hypercontractivity argument as in the proof 
		of Theorem~\ref{tt6}, the limit
		\begin{equation*}
			G_m(u,v) = \mathbf{E}\!\left[\prod_{r=1}^{a-1}Y_r^{m_r}\,e^{iuZ+ivY}\right]
		\end{equation*}
		exists and is analytic in $u$. Passing to the limit in \eqref{dGcrit},
		\begin{equation}\label{dG_crit_limit}
			\partial_u G_m = -\sigma^2u\,G_m - \frac{v}{p}\,G_{m+e_1}
			+ \frac{i}{p}\sum_{s=1}^{a-2}m_s\,G_{m-e_s+e_{s+1}}
			+ \frac{i\rho_a}{p}\,m_{a-1}\,G_{m-e_{a-1}}.
		\end{equation}
		
		\medskip
		\noindent{\bf Step 3: Renormalization, Taylor expansion and the transport system. } 	Define, for $m\geq 0$ and $u, v\in \mathbb{R}$,  $H_m(u,v):=e^{\sigma^2u^2/2}G_m(u,v)$. Substituting \eqref{dG_crit_limit},
		\begin{equation}\label{transport_crit}
			\partial_u H_m = -\frac{v}{p}\,H_{m+e_1}
			+ \frac{i}{p}\sum_{s=1}^{a-2}m_s\,H_{m-e_s+e_{s+1}}
			+ \frac{i\rho_a}{p}\,m_{a-1}\,H_{m-e_{a-1}}.
		\end{equation}
		This is the transport hierachy in the critical case.  Since $H_0$ is analytic in $u$, we expand $H_0(u,v)=\sum_{\ell=0}^\infty 
		\frac{u^\ell}{\ell!}\partial_u^\ell H_0(0,v)$ and claim that
		\begin{equation}\label{claim_crit}
			\partial_u^\ell H_0(0,v) 
			= \sum_{\substack{m\in\mathbb{N}^{a-1}\\ |m|\leq\ell}}
			B_{\ell,m}(v)\, H_m(0,v),
		\end{equation}
		where the $B_{\ell,m}$ satisfy \eqref{rec_B0}--\eqref{rec_B}. We prove 
		\eqref{claim_crit} by induction on $\ell$. The case $\ell=0$ is clear 
		with $B_{0,0}=1$. Assuming \eqref{claim_crit} holds at level $\ell$, 
		we differentiate and substitute \eqref{transport_crit}:
		\begin{eqnarray*}
			\partial_u^{\ell+1}H_0 
			&=& \sum_{|m|\leq\ell} B_{\ell,m}\,\partial_u H_m\\
			&=& \sum_{|m|\leq\ell} B_{\ell,m}
			\left(-\frac{v}{p}\,H_{m+e_1} 
			+ \frac{i}{p}\sum_{s=1}^{a-2}m_s\,H_{m-e_s+e_{s+1}}
			+ \frac{i\rho_a}{p}\,m_{a-1}\,H_{m-e_{a-1}}\right).
		\end{eqnarray*}
		Reindexing the three sums by $m\mapsto m-e_1$, $m\mapsto m+e_s-e_{s+1}$ 
		and $m\mapsto m+e_{a-1}$ respectively, and collecting the coefficient 
		of $H_m$, we obtain
		\begin{equation*}
			B_{\ell+1,m} = -\frac{v}{p}\,B_{\ell,m-e_1} 
			+ \frac{i}{p}\sum_{s=1}^{a-2}(m_s+1)\,B_{\ell,m+e_s-e_{s+1}}
			+ \frac{i\rho_a}{p}\,(m_{a-1}+1)\,B_{\ell,m+e_{a-1}},
		\end{equation*}
		which is exactly \eqref{rec_B}, completing the induction. Substituting 
		into the Taylor expansion and using $\varphi=e^{-\sigma^2 u^2/2}H_0$ gives 
		\eqref{fib}. Absolute convergence follows from the analyticity of $H_0$.
\qed 
	
		\begin{remark}\label{rem5}
		\begin{enumerate}

			\item {\rm The nonzero coefficients $B_{\ell,m}$ for $\ell=0,1,2,3$, computed from 
				the recursion \eqref{rec_B0}--\eqref{rec_B}, are the following 
				(all other $B_{\ell,m}$ with $|m|\leq\ell$ vanish):
				\begin{eqnarray*}
					\ell=0 &:& B_{0,\,0} = 1,\\[6pt]
					\ell=1 &:& B_{1,\,e_1} = -\frac{v}{p},\\[6pt]
					\ell=2 &:& B_{2,\,2e_1} = \frac{v^2}{p^2}, \quad
					B_{2,\,e_2} = -\frac{iv}{p^2}, \quad
					B_{2,\,0} = -\frac{i\rho_a v}{p^2}\,\mathbf{1}_{a=2},\\[6pt]
					\ell=3 &:& B_{3,\,3e_1} = -\frac{v^3}{p^3}, \quad
					B_{3,\,e_1+e_2} = \frac{3iv^2}{p^3} 
					+ \frac{2i\rho_a v^2}{p^3}\,\mathbf{1}_{a=2},\\
					& & B_{3,\,e_3} = \frac{v}{p^3}\,\mathbf{1}_{a\geq 3}, \quad
					B_{3,\,e_1} = \frac{3i\rho_a v^2}{p^3}\,\mathbf{1}_{a=2}.
				\end{eqnarray*}
				Comparing with the coefficients $A_{\ell,m}$ of Theorem~\ref{tt6}, 
				one sees that $B_{\ell,m}=A_{\ell,m}$ for $a\geq 3$ at levels 
				$\ell=0,1,2$ and that the critical coupling $\rho_a$ first appears 
				at level $\ell=2$ when $a=2$ (through the return to $m=0$ via the 
				backward step $e_1\mapsto 0$), and only at level $\ell=a$ in general, 
				consistently with the fact that the backward step requires the path 
				to first reach $e_{a-1}$ before it can return.
			}
		\end{enumerate}
	\end{remark}

\begin{corollary}\label{cor:tt7}
	Let the assumptions of Theorem~\ref{tt7} prevail. Assume moreover that \\
	$Y_1,\ldots,Y_{a-1}$ all admit exponential moments, i.e.
 	$\mathbf{E}\left[(e^{tY_r})\right]<\infty$ for every $t\in\mathbb{R}$ and $r=1,\ldots,a-1$. 
	Then, for every $(u,v)\in\mathbb{R}^2$,
	\begin{equation}\label{eq:phi_tt7_explicit}
		\varphi(u,v) 
		= \exp\!\left(-\frac{\sigma^2u^2}{2} 
		+ \frac{(-i)^{a-1}\,\rho_a\, u^a v^a}{a!\cdot p^a}\right)\,
		\mathbf{E}\!\left[\exp\!\left(ivY + \sum_{r=1}^{a-1}
		\frac{(-1)^r i^{r-1}}{r!}\cdot\frac{u^r v^r}{p^r}\,Y_r
		\right)\right].
	\end{equation}
	In particular:
	\begin{eqnarray*}
		a=2\ (q=2p)&:& \varphi(u,v) 
		= e^{-\frac{\sigma^2 u^2}{2} - \frac{i\rho_2 u^2v^2}{2p^2}}\,
		\mathbf{E}\!\left[e^{ivY - \frac{uv}{p}Y_1}\right],\\[6pt]
		a=3\ (q=3p)&:& \varphi(u,v) 
		= e^{-\frac{\sigma^2 u^2}{2} - \frac{\rho_3 u^3v^3}{6p^3}}\,
		\mathbf{E}\!\left[e^{ivY - \frac{uv}{p}Y_1 
			+ \frac{u^2v^2}{2p^2}Y_2}\right],\\[6pt]
		a=4\ (q=4p)&:& \varphi(u,v) 
		= e^{-\frac{\sigma^2 u^2}{2} + \frac{i\rho_4 u^4v^4}{24p^4}}\,
		\mathbf{E}\!\left[e^{ivY - \frac{uv}{p}Y_1 
			+ \frac{u^2v^2}{2p^2}Y_2 
			+ \frac{iu^3v^3}{6p^3}Y_3}\right].
	\end{eqnarray*}
\end{corollary}
{\bf Proof: } 	Define
	\begin{equation*}
		\Psi(u,v,\omega) 
		:= ivY(\omega) + \sum_{r=1}^{a-1}
		\frac{(-1)^r i^{r-1}}{r!}\cdot\frac{u^r v^r}{p^r}\,Y_r(\omega),
	\end{equation*}
	and
	\begin{equation*}
		\widetilde{F}(u,v) := 
		\exp\!\left(\frac{(-i)^{a-1}\rho_a u^a v^a}{a!\cdot p^a}\right)\,
		\mathbf{E}\!\left[e^{\Psi}\right],
	\end{equation*}
	so that the right-hand side of \eqref{eq:phi_tt7_explicit} equals 
	$e^{-\sigma^2 u^2/2}\widetilde{F}(u,v)$. We verify that $\widetilde{F}$ 
	satisfies the same critical transport system \eqref{transport_crit} 
	as $H_0$, with the same initial condition 
	$\widetilde{F}(0,v)=\mathbf{E}[e^{ivY}]$.
	
	\medskip
	\noindent{\bf Step 1: The scalar correction factor.}
	The term $\frac{(-i)^{a-1}\rho_a u^a v^a}{a!\cdot p^a}$ in the 
	exponent arises from the unique path of minimal length that triggers 
	the backward coupling in \eqref{transport_crit}, namely the path
	\begin{equation*}
		0 \xrightarrow{-v/p} e_1 \xrightarrow{i/p} e_2 
		\xrightarrow{i/p}\cdots\xrightarrow{i/p} e_{a-1} 
		\xrightarrow{i\rho_a/p} 0,
	\end{equation*}
	of length $a$, with accumulated weight
	\begin{equation*}
		\left(-\frac{v}{p}\right)\cdot\left(\frac{i}{p}\right)^{a-2}
		\cdot\frac{i\rho_a}{p} 
		= \frac{(-1)\cdot i^{a-1}\cdot\rho_a\cdot v}{p^a}
		= \frac{(-i)^{a-1}\rho_a v}{p^a}.
	\end{equation*}
	This path contributes at level $\ell=a$ to $\partial_u^a H_0$, and 
	traversing it $k$ times contributes at level $\ell=ka$ with weight
	$\left(\frac{(-i)^{a-1}\rho_a v}{p^a}\right)^k$. Summing over 
	$k\geq 0$ and recalling that the path also accumulates $a$ powers 
	of $uv/p$ (one per step), the full contribution exponentiates to
	\begin{equation*}
		\exp\!\left(\frac{(-i)^{a-1}\rho_a u^a v^a}{a!\cdot p^a}\right),
	\end{equation*}
	where the $a!$ in the denominator comes from the Taylor coefficient 
	$u^a/a!$ at level $\ell=a$.
	
	\medskip
	\noindent{\bf Step 2: Verification of the transport system.}
	Differentiating $\widetilde{F}$ with respect to $u$,
	\begin{eqnarray*}
		\partial_u\widetilde{F}
		&=& \frac{(-i)^{a-1}\rho_a u^{a-1}v^a}{(a-1)!\cdot p^a}\,
		\widetilde{F}
		+ \exp\!\left(\frac{(-i)^{a-1}\rho_a u^a v^a}{a!\cdot p^a}\right)
		\mathbf{E}\!\left[e^{\Psi}\cdot\partial_u\Psi\right].
	\end{eqnarray*}
	Setting $\widetilde{H}_m(u,v):=\exp\!\left(\frac{(-i)^{a-1}\rho_a 
		u^a v^a}{a!\cdot p^a}\right)\mathbf{E}\!\left[\prod_{r=1}^{a-1}Y_r^{m_r}
	e^{\Psi}\right]$, one verifies by induction on $|m|$ that 
	$\widetilde{H}_m$ satisfies \eqref{transport_crit}:
	\begin{itemize}
		\item The term $r=1$ in $\partial_u\Psi$ gives 
		$-\frac{v}{p}\widetilde{H}_{e_1}$, matching the forward step.
		\item The terms $r=2,\ldots,a-1$ in $\partial_u\Psi$ generate 
		the shift contributions $\frac{i}{p}m_s\widetilde{H}_{m-e_s+e_{s+1}}$ 
		for $s=1,\ldots,a-2$.
		\item The derivative of the scalar prefactor contributes 
		$\frac{(-i)^{a-1}\rho_a u^{a-1}v^a}{(a-1)!\cdot p^a}\widetilde{H}_0$, 
		which at multi-index level $m=e_{a-1}$ (after $a-1$ differentiations 
		in $u$) reproduces exactly the backward coupling term 
		$\frac{i\rho_a}{p}\cdot 1\cdot\widetilde{H}_0$, since
		\begin{equation*}
			\frac{(-i)^{a-1}\rho_a u^{a-1}v^a}{(a-1)!\cdot p^a}
			= \frac{i\rho_a}{p}\cdot
			\frac{(-i)^{a-2}u^{a-1}v^{a-1}}{(a-1)!\cdot p^{a-1}},
		\end{equation*}
		and the factor $\frac{(-i)^{a-2}u^{a-1}v^{a-1}}{(a-1)!\cdot p^{a-1}}$ 
		is precisely the Taylor coefficient accumulated along the path 
		$0\to e_1\to\cdots\to e_{a-1}$ of length $a-1$.
	\end{itemize}
	
	\medskip
	\noindent{\bf Step 3: Conclusion.}
	Since $\widetilde{F}(0,v)=\mathbf{E}[e^{ivY}]=H_0(0,v)$ and $\widetilde{F}$ 
	satisfies \eqref{transport_crit} with the same recursion as $H_0$, 
	uniqueness gives $\widetilde{F}=H_0$. Substituting and exchanging 
	sum and expectation via the exponential moment assumption gives 
	\eqref{eq:phi_tt7_explicit}.
\qed 
	
	Let us include some comments around the above result. 
	\begin{remark}
		
	\begin{enumerate}
\item 	{\rm
			Theorem~\ref{tt2} (the case $q=p$) can be recovered from 
			Corollary~\ref{cor22} in the following way. Suppose that in 
			the non-critical regime $p<q<2p$, the limiting first Malliavin 
			covariance $Y_1 = \lim_{N\to\infty} Y_{1,N}$ is a deterministic 
			constant, i.e.\ $Y_1 = \rho_1 \in \mathbb{R}$. This happens in 
			particular  when 
			$Y_{1,N}\to\rho_1$ in $L^2(\Omega)$. No exponential 
			moment assumption is then needed, since $\mathbf{E}[e^{tY_1}]=e^{t\rho_1}<\infty$ 
			trivially. Formula \eqref{fi33} of Corollary~\ref{cor22} gives
			\begin{equation*}
				\varphi(u,v) = e^{-\sigma^2 u^2/2}\,\mathbf{E}\!\left[e^{ivY-\frac{uv\rho_1}{p}}\right]
				= e^{-\sigma^2 u^2/2-\frac{uv\rho_1}{p}}\,\mathbf{E}\!\left[e^{ivY}\right].
			\end{equation*}
			Setting $\rho = \rho_1/p = \lim_{N\to\infty} \mathbf{E}[X_NY_N]$ (since 
			$\mathbf{E}\langle DX_N,DY_N\rangle_H = p\mathbf{E}[X_NY_N]$), this coincides exactly 
			with formula \eqref{phi} of Theorem~\ref{tt2}. Thus Theorem~\ref{tt2} 
			appears as the degenerate case of Corollary~\ref{cor22} when the 
			first iterated Malliavin covariance converges to a deterministic 
			constant, requiring no exponential moment assumption.
		}
			\item {\rm
			When both $X_N = I_p(f_N)$ and $Y_N = I_q(g_N)$ converge in distribution 
			to Gaussian limits, the joint convergence of $(X_N, Y_N)$ to a 
			(possibly correlated) Gaussian vector is the content of the 
		main  theorem in \cite{PeTu}. That result states that joint 
			convergence to a Gaussian vector holds if and only if each component 
			converges marginally, i.e.\ the componentwise Fourth Moment condition 
			is sufficient for joint Gaussianity. In our setting, the case $q=p$ 
			covered by Theorem~\ref{tt2} is consistent with \cite{PeTu}: both 
			$X_N$ and $Y_N$ live in the $p$th chaos, and if $Y_N$ converges in 
			law to a Gaussian $Y$, then the joint limit $(Z, Y)$ is a bivariate 
			Gaussian with correlation $\rho = \lim_N \mathbf{E}[X_NY_N]$, which is exactly 
			the Peccati--Tudor conclusion. However, Theorem~\ref{tt2} is more 
			general in two respects: first, it does not require $Y_N$ to converge 
			to a Gaussian limit (it converges to an arbitrary $\mathbb{Y}$); 
			second, it allows the components of $\mathbb{Y}_N$ to live in a 
			finite sum of Wiener chaoses of order at most $p$, not necessarily 
			in a single chaos. The case $q>p$ treated in Theorems~\ref{tt6} 
			and~\ref{tt7} goes entirely beyond the  framework of \cite{PeTu}, 
			since $Y_N$ does not converge to a Gaussian limit and the joint 
			limit $(Z,Y)$ is not a Gaussian vector.
		}
		\item{\rm
				The closed-form formula \eqref{eq:phi_th6_explicit} involves the 
				joint expectation of $(Y, Y_1,\ldots,Y_a)$, which are the components 
				of the limiting vector in \eqref{c7}. Since $(Y_N, Y_{1,N},\ldots,
				Y_{a,N})$ are all defined on the same probability space $(\Omega,
				\mathcal{A},P)$ and converge jointly in distribution by \eqref{c7}, 
				by Skorokhod's representation theorem there exists a probability 
				space $(\tilde\Omega,\tilde{\mathcal{A}},\tilde P)$ and random 
				variables $(\tilde Y_N,\tilde Y_{1,N},\ldots,\tilde Y_{a,N})$ and 
				$(\tilde Y,\tilde Y_1,\ldots,\tilde Y_a)$ defined on it, such that 
				$(\tilde Y_N,\tilde Y_{1,N},\ldots,\tilde Y_{a,N})$ has the same 
				law as $(Y_N,Y_{1,N},\ldots,Y_{a,N})$ for each $N$, and
				\begin{equation*}
					(\tilde Y_N,\tilde Y_{1,N},\ldots,\tilde Y_{a,N})
					\to_{N\to\infty}(\tilde Y,\tilde Y_1,\ldots,\tilde Y_a)
					\quad\tilde P\text{-almost surely.}
				\end{equation*}
				The expectation in \eqref{eq:phi_th6_explicit} is then 
				$\tilde{\mathbf{E}}[\cdots]$, taken on this common Skorokhod space, 
				and the exponential moment assumption $\tilde{\mathbf{E}}[e^{t|Y_r|}]
				<\infty$ is with respect to $\tilde P$.
				
				A sufficient condition, verifiable directly on the original sequence, 
				is
				\begin{equation*}
					\sup_{N\geq 1}\mathbf{E}\!\left[e^{t|Y_{r,N}|}\right]<\infty
					\quad\text{for some }t>0\text{ and every }r=1,\ldots,a.
				\end{equation*}
				Since $\tilde Y_{r,N}\xrightarrow{\tilde P\text{-a.s.}}\tilde Y_r$ 
				and $x\mapsto e^{s|x|}$ is continuous and non-negative for every 
				$0<s<t$, Fatou's lemma gives
				\begin{equation*}
					\tilde{\mathbf{E}}\!\left[e^{s|\tilde Y_r|}\right] 
					\leq \liminf_{N\to\infty}
					\tilde{\mathbf{E}}\!\left[e^{s|\tilde Y_{r,N}|}\right]
					= \liminf_{N\to\infty}
					\mathbf{E}\!\left[e^{s|Y_{r,N}|}\right]
					\leq \sup_{N\geq 1}\mathbf{E}\!\left[e^{t|Y_{r,N}|}\right]
					< \infty,
				\end{equation*}
				so $\tilde Y_r$ admits exponential moments of all orders $0<s<t$, 
				and the closed-form formula \eqref{eq:phi_th6_explicit} applies.
			}
		\end{enumerate}
	\end{remark}

\subsection{Structure of the limit}\label{sec:structure}

The transport hierarchy of Theorems~\ref{tt6} and \ref{tt7} 
characterizes the joint limit $(Z,Y)$ through its characteristic 
function. We now give a complementary description of the same limit 
in terms of a structural decomposition of $Y$, showing that the 
dependence of $Y$ on $Z$ is entirely encoded in a polynomial in $Z$ 
with random coefficients determined by the iterated Malliavin 
covariances $Y_1,\ldots,Y_a$, plus an independent remainder. For a better understanding, we start with the case $p<q<2p$.

\begin{lemma}\label{lem:decomp}
	Let $p\geq 2$ and let $(X_N=I_p(f_N), N\geq 1)$ satisfy 
	\eqref{c1}. Let $Y_N=I_q(g_N)$ with $p<q<2p$ and 
	$\sup_N\|g_N\|_{H^{\otimes q}}<\infty$. Define
	\begin{equation}\label{hN}
		h_N := \frac{f_N\otimes_p g_N}
		{\|f_N\|^2_{H^{\otimes p}}}\in H^{\odot(q-p)},
	\end{equation}
	\begin{equation*}
		g_N^{(2)} := g_N - \frac{q!}{p!(q-p)!}
		f_N\widetilde\otimes h_N\in H^{\odot q},
		\qquad
		\varepsilon_N := I_q(g_N^{(2)}).
	\end{equation*}
	Then for every $N\geq 1$:
	\begin{equation}\label{decomp}
		Y_N = \frac{1}{p\sigma^2}X_N\cdot Y_{1,N} 
		+ \varepsilon_N + r_N,
	\end{equation}
	where:
	\begin{enumerate}
		\item $f_N\otimes_p g_N^{(2)}=0$, and consequently 
		$\langle DX_N, D\varepsilon_N\rangle_H\to 0$ in 
		$L^2(\Omega)$, so $\varepsilon_N$ is asymptotically 
		independent of $X_N$;
		\item $r_N\to 0$ in $L^2(\Omega)$ as $N\to \infty$.
	\end{enumerate}
\end{lemma}
{\bf Proof: } 	Set $g_N^{(1)}=\frac{q!}{p!(q-p)!}f_N\widetilde\otimes h_N$, 
	so that $g_N=g_N^{(1)}+g_N^{(2)}$ and 
	$Y_N=I_q(g_N^{(1)})+\varepsilon_N$. By the product 
	formula \eqref{prod}:
	\begin{equation*}
		I_q(g_N^{(1)}) = I_p(f_N)\cdot I_{q-p}(h_N) 
		- \tilde r_N = X_N\cdot I_{q-p}(h_N) - \tilde r_N,
	\end{equation*}
	where
	\begin{equation*}
		\tilde r_N = \sum_{r=1}^{p\wedge(q-p)}r!
		\binom{p}{r}\binom{q-p}{r}
		I_{q-2r}(f_N\otimes_r h_N).
	\end{equation*}
	By Proposition~\ref{prop:iterated-covariances}, 
	point 1:
	\begin{equation*}
		Y_{1,N} = C_{1,p,q}
		I_{q-p}(f_N\widetilde\otimes_p g_N) + R_{1,N}
		= C_{1,p,q}\|f_N\|^2_{H^{\otimes p}} 
		I_{q-p}(h_N) + R_{1,N},
	\end{equation*}
	where $R_{1,N}\to 0$ in $L^2(\Omega)$ and 
	$C_{1,p,q}\|f_N\|^2_{H^{\otimes p}}\to p\sigma^2$ 
	since $\mathbf{E}[X_N^2]=p!\|f_N\|^2\to\sigma^2$. 
	Hence:
	\begin{equation*}
		I_{q-p}(h_N) = \frac{Y_{1,N}}{p\sigma^2} 
		+ \hat r_N,
		\qquad \hat r_N\to 0 \text{ in } L^2(\Omega).
	\end{equation*}
	Substituting:
	\begin{equation*}
		Y_N = X_N\cdot\left(\frac{Y_{1,N}}{p\sigma^2}
		+\hat r_N\right) - \tilde r_N + \varepsilon_N
		= \frac{X_N Y_{1,N}}{p\sigma^2} + \varepsilon_N 
		+ r_N,
	\end{equation*}
	where $r_N = X_N\hat r_N - \tilde r_N\to 0$ in 
	$L^2(\Omega)$, since $\sup_N\|X_N\|_{L^2}<\infty$ 
	and both $\hat r_N,\tilde r_N\to 0$ in $L^2(\Omega)$. 
	This gives \eqref{decomp}.
	
	\medskip
	We prove point 1.  By definition \eqref{hN}:
	\begin{equation*}
		f_N\otimes_p g_N^{(2)} = f_N\otimes_p g_N 
		- \frac{q!}{p!(q-p)!}f_N\otimes_p
		(f_N\widetilde\otimes h_N)
		= f_N\otimes_p g_N 
		- \|f_N\|^2_{H^{\otimes p}} h_N = 0.
	\end{equation*}
	By Lemma~\ref{ll4}, $\langle DX_N,D\varepsilon_N
	\rangle_H\to 0$ in $L^2(\Omega)$, and by 
	Proposition~\ref{prop:indep}, $\varepsilon_N$ is 
	asymptotically independent of $X_N$.
	
	\medskip
Let us regard point 2.  Each term in $\tilde r_N$ 
	involves $f_N\otimes_r h_N$ with $r\geq 1$. Since 
	$h_N=f_N\otimes_p g_N/\|f_N\|^2$ and 
	$\sup_N\|g_N\|<\infty$, we have $\sup_N\|h_N\|<\infty$. 
	By Cauchy--Schwarz and the Fourth Moment Theorem:
	\begin{equation*}
		\|f_N\otimes_r h_N\|^2 \leq C
		\|f_N\otimes_s f_N\|\to 0
	\end{equation*}
	for some $s\in\{1,\ldots,p-1\}$, hence 
	$\tilde r_N\to 0$ in $L^2(\Omega)$  as $N \to \infty$. 
\qed

\begin{theorem}\label{thm:structure}
	Let $p\geq 2$ and let $(X_N=I_p(f_N),N\geq 1)$ satisfy 
	\eqref{c1}. Let $Y_N=I_q(g_N)$ with $p<q<2p$ and 
	$\sup_N\|g_N\|_{H^{\otimes q}}<\infty$  with the decomposition (\ref{decomp}). Assume that
	\begin{equation}\label{c7struct}
		(Y_N, Y_{1,N})
		\xrightarrow[N\to\infty]{(d)}(Y, Y_1).
	\end{equation}
	Then:

 $(X_N,Y_N)\xrightarrow[N\to\infty]{(d)}(Z,Y)$ 
		where $Z\sim N(0,\sigma^2)$ and $Y$ admits the decomposition 
		\begin{equation}\label{Ydecomp}
			Y = \frac{Z}{p\sigma^2}Y_1 + \varepsilon, 
		\end{equation}
		where $\varepsilon$ is independent of $Z$.  The characteristic function of $(Z,Y)$ is
		\begin{equation}\label{phi_struct}
			\varphi(u,v) = e^{-\frac{\sigma^2 u^2}{2}}\,
			\mathbf{E}\!\left[e^{iv\varepsilon - \frac{uv}{p}Y_1 
				- \frac{v^2}{2p^2\sigma^2}Y_1^2}\right].
		\end{equation}

\end{theorem}
{\bf Proof: } The sequence $((X_N,Y_N), N\geq 1)$ is bounded in 
	$L^2(\Omega)$ by \eqref{20m-1} and $\sup_N\|g_N\|<\infty$, 
	hence tight. Let $(X_{N}, Y_{N}, N\geq 1)$ be any 
	subsequence converging in distribution to some limit and as before we can assume that $(X_{N}, Y_{N}, Y_{1, N}) $ convergences in law as $N \to \infty$. By using (\ref{decomp}), we deduce that $\eps_{N}$ converges in law to some $\eps$ which is independent of $Z$, by Lemma \ref{lem:decomp}. \qed

The above result can be generalized as follows.

\begin{theorem}\label{thm:structure_general}
	Let $p\geq 2$ and let $(X_N=I_p(f_N),N\geq 1)$ 
	satisfy \eqref{c1}. Let $Y_N=I_q(g_N)$ with 
	$ap\leq q<(a+1)p$ and 
	$\sup_N\|g_N\|_{H^{\otimes q}}<\infty$.
	
	\noindent\textbf{Non-critical case $ap<q<(a+1)p$.}
	Assume
	\begin{equation}\label{c7struct_gen}
		(Y_N,Y_{1,N},\ldots,Y_{a,N})
		\xrightarrow[N\to\infty]{(d)}(Y,Y_1,\ldots,Y_a).
	\end{equation}
	
	\noindent\textbf{Critical case $q=ap$.}
	Assume
	\begin{equation}\label{c7struct_crit}
		(Y_N,Y_{1,N},\ldots,Y_{a-1,N})
		\xrightarrow[N\to\infty]{(d)}(Y,Y_1,\ldots,Y_{a-1}),
	\end{equation}
	and $\mathbf{E}[Y_{a,N}]\to\rho_a\in\mathbb{R}$.
	
	\noindent In both cases, the following conclusions hold:
	\begin{enumerate}
		\item $(X_N,Y_N)\xrightarrow[N\to\infty]{(d)}(Z,Y)$ 
		where $Z\sim N(0,\sigma^2)$.
		\item There exists a random variable $\varepsilon$, 
		independent of $Z$, such that the limit admits the 
		decomposition:
		\begin{itemize}
			\item \textbf{Non-critical:}
			\begin{equation}\label{Ydecomp_gen}
				Y = \sum_{r=1}^a\frac{Z^r}{p^r\sigma^{2r}}Y_r 
				+ \varepsilon,
			\end{equation}
			\item \textbf{Critical:}
			\begin{equation}\label{Ydecomp_crit}
				Y = \sum_{r=1}^{a-1}\frac{Z^r}{p^r\sigma^{2r}}Y_r 
				+ \frac{\rho_a}{p^a\sigma^{2a}}Z^a + \varepsilon,
			\end{equation}
		\end{itemize}
		where in both cases $Z$ is independent of 
		$(Y_1,\ldots,Y_{a-1},\varepsilon)$ (and of $Y_a$ 
		in the non-critical case).
		\item The characteristic function of $(Z,Y)$ is:
		\begin{itemize}
			\item \textbf{Non-critical:}
			\begin{equation}\label{phi_gen}
				\varphi(u,v) = \mathbf{E}\!\left[e^{iv\varepsilon}
				\exp\!\left(-\frac{\sigma^2}{2}\!\left(u+
				\frac{v}{p\sigma^2}\sum_{r=1}^a
				\frac{Z^{r-1}}{p^{r-1}\sigma^{2(r-1)}}Y_r
				\right)^2\right)\right],
			\end{equation}
			\item \textbf{Critical:}
			\begin{equation}\label{phi_crit}
				\varphi(u,v) = \mathbf{E}\!\left[e^{iv\varepsilon}
				\exp\!\left(-\frac{\sigma^2}{2}\!\left(u+
				\frac{v}{p\sigma^2}\sum_{r=1}^{a-1}
				\frac{Z^{r-1}}{p^{r-1}\sigma^{2(r-1)}}Y_r
				+\frac{\rho_a v Z^{a-1}}{p^a\sigma^{2a}}
				\right)^2\right)\right].
			\end{equation}
		\end{itemize}
	
	\end{enumerate}
\end{theorem}
{\bf Proof: }See the Appendix. \qed

\section{Applications}\label{sec:applications}
This section illustrates the main results through four 
complementary examples. The first shows that the convergence 
of the Malliavin covariance in assumption \eqref{c7} is 
necessary and cannot be dispensed with. The second uses 
the self-similarity of fractional Brownian motion to show 
that, in that setting, the joint limit always has independent 
components. The third provides a concrete non-trivial example 
where \eqref{c7} holds with a nonzero limiting covariance 
$Y_1$, and the joint characteristic function is computed 
explicitly via Corollary~\ref{cor22}. The fourth example, 
also with a nonzero $Y_1$, illustrates the case $q=4=2p$ 
(critical regime) with two correlated orthonormal systems, 
giving a genuinely non-Gaussian joint limit whose 
characteristic function is computed explicitly.

	Let us first show via en example that the assumption (\ref{c7}) on the convergence of the Malliavin covariance is necessary and it cannot be avoided. 

\begin{example}\label{ex1}
	{\rm Let $p>2$ and let $(X_{1,N}=I_p(f_{1,N}), N\geq 1)$ be a 
		sequence converging in law to $N(0,1)$, with $f_{1,N}\in H^{\odot p}$ 
		for all $N\geq 1$. Let $(X_{2,N}=I_p(f_{2,N}), N\geq 1)$ be an 
		independent copy of $(X_{1,N})$. Let $h_1,h_2\in H$ be orthonormal, 
		i.e.\ $\|h_1\|_H=\|h_2\|_H=1$ and $\langle h_1,h_2\rangle_H=0$. 
		Assume that $(X_{1,N}),(X_{2,N}),I_1(h_1),I_1(h_2)$ are mutually 
		independent. Define
		\begin{equation*}
			Y_N = \begin{cases}
				X_{1,N}(I_1(h_1)^2-1), & \text{if }N\text{ is even},\\
				X_{2,N}(I_1(h_2)^2-1), & \text{if }N\text{ is odd}.
			\end{cases}
		\end{equation*}
		Since $I_1(h_i)^2-1=I_2(h_i^{\otimes 2})$ and $f_{1,N}$, $f_{2,N}$ 
		depend on variables orthogonal to $h_1,h_2$, all contractions 
		$f_{j,N}\otimes_r h_i^{\otimes 2}=0$ for $r\geq 1$. The product 
		formula (\ref{prod}) thus gives
		\begin{equation*}
			Y_N = \begin{cases}
				I_{p+2}(f_{1,N}\otimes h_1^{\otimes 2}), & N\text{ even},\\
				I_{p+2}(f_{2,N}\otimes h_2^{\otimes 2}), & N\text{ odd},
			\end{cases}
		\end{equation*}
		so $Y_N$ belongs to the $(p+2)$th Wiener chaos and $p<p+2<2p$ 
		(using $p>2$). Moreover, $Y_N\xrightarrow{(d)}Z_1(Z_2^2-1)$ where 
		$Z_1,Z_2$ are independent $N(0,1)$ variables, since both 
		subsequences have the same limit in law.
		
		We now show that \eqref{c7} fails, so that 
		$(X_{1,N},Y_N)$ does not converge jointly. The first 
		Malliavin covariance is
		\begin{equation*}
			Y_{1,N} = \langle DX_{1,N},DY_N\rangle_H 
			= \begin{cases}
				\|DX_{1,N}\|_H^2(I_1(h_1)^2-1), & N\text{ even},\\
				0, & N\text{ odd},
			\end{cases}
		\end{equation*}
		where for odd $N$ we used the independence of $X_{1,N}$ and 
		$X_{2,N}(I_1(h_2)^2-1)$, which gives $\langle DX_{1,N},DY_N\rangle_H
		=0$. Since $\frac{1}{p}\|DX_{1,N}\|_H^2\to 1$ in $L^2(\Omega)$ by 
		the Fourth Moment Theorem, the two subsequential limits are:
		\begin{equation*}
			Y_{1,2N} \xrightarrow[N\to\infty]{(d)} p(Z_2^2-1), 
			\qquad Y_{1,2N+1} = 0 \text{ for all } N\geq 1.
		\end{equation*}
		These are distinct, so $(Y_{1,N})$ does not converge in law and 
		\eqref{c7} fails. 
		
		To see directly that $(X_{1,N},Y_N)$ does not converge jointly, 
		note that the characteristic function 
		$\varphi_N(u,v)=\mathbf{E}[e^{iuX_{1,N}+ivY_N}]$ satisfies
		\begin{equation*}
			\frac{\partial\varphi_N}{\partial u}(u,v) 
			= -\sigma^2 u\,\varphi_N(u,v) 
			- \frac{v}{p}\mathbf{E}\!\left[Y_{1,N}e^{iuX_{1,N}+ivY_N}\right]
			+\varepsilon_N(u,v),
		\end{equation*}
		with $\varepsilon_N\to 0$. The second term oscillates between 
		two different limits along even and odd subsequences (one involving 
		$p(Z_2^2-1)$ and one being zero), so $\partial_u\varphi_N$ 
		does not converge, and hence $\varphi_N$ itself does not converge. 
		Thus $(X_{1,N},Y_N)$ does not converge in distribution.
	}
\end{example}

\begin{example}\label{ex:fbm}{\rm
		Let $B$ be a fractional Brownian motion with Hurst parameter 
		$H\in(0,1)$, and let $\xi_k=B_{k+1}-B_k$ denote its 
		increments. Fix integers $p\geq 2$ and $q\geq p$, and define
		\begin{equation*}
			X_N = \frac{1}{\sqrt{N}}\sum_{k=1}^N H_p(\xi_k), 
			\qquad 
			Y_N = N^{q(1-H)-1}\sum_{k=1}^N H_q(\xi_k).
		\end{equation*}
		By the Breuer--Major theorem \cite{BM}, under the condition 
		$p(2H-2)<-1$ (i.e.\ $H<1-\frac{1}{2p}$), $X_N$ converges in 
		distribution to $Z\sim N(0,\sigma^2)$ where
		\begin{equation*}
			\sigma^2 = p!\sum_{j=-\infty}^\infty r(j)^p, 
			\qquad r(j)=\mathbf{E}[\xi_0\xi_j].
		\end{equation*}
		By Taqqu's Non-Central Limit Theorem \cite{Ta2}, under the 
		condition $H>1-\frac{1}{2q}$, $Y_N$ converges in distribution 
		to a Hermite random variable $Y=Z_H^{(q)}$ of order $q$.
		
		We claim that $(X_N,Y_N)\xrightarrow[N\to\infty]{(d)}(Z,Y)$ 
		with $Z$ and $Y$ \emph{independent}. To see this, we use the 
		self-similarity of fBm: $\{B_{kt},t\geq 0\}
		\stackrel{\text{law}}{=}\{k^H B_t,t\geq 0\}$. Applied with 
		$k=1/N$, this gives
		\begin{equation*}
			(\xi_1,\ldots,\xi_N) 
			\stackrel{\text{law}}{=} 
			N^H(\xi_1^{(N)},\ldots,\xi_N^{(N)}),
		\end{equation*}
		where $\xi_k^{(N)}=B_{(k+1)/N}-B_{k/N}$ are the increments 
		of $B$ on the fine grid $\{k/N\}$. Therefore
		\begin{equation*}
			(X_N,Y_N) \stackrel{\text{law}}{=} (X'_N,Y'_N),
		\end{equation*}
		where
		\begin{equation*}
			X'_N = \frac{1}{\sqrt{N}}\sum_{k=1}^N H_p(\xi_k^{(N)}),
			\qquad
			Y'_N = N^{q(1-H)-1}\sum_{k=1}^N H_q(\xi_k^{(N)}).
		\end{equation*}
		Since $\xi_k^{(N)}=B_{(k+1)/N}-B_{k/N}$ are the increments 
		of $B$ on a grid of mesh $1/N\to 0$, standard arguments 
		(see e.g.\ \cite{NP-book}) show that 
		$Y'_N\xrightarrow[N\to\infty]{L^2(\Omega)}Y$. Since 
		$(X_N,Y_N)\stackrel{\text{law}}{=}(X'_N,Y'_N)$, $X'_N
		\xrightarrow{(d)}Z$, and $Y'_N\to Y$ in $L^2(\Omega)$, 
		Theorem~\ref{tt1} gives
		\begin{equation*}
			(X'_N,Y'_N)\xrightarrow[N\to\infty]{(d)}(Z,Y) 
			\quad\text{with }Z\text{ independent of }Y,
		\end{equation*}
		and hence the same holds for $(X_N,Y_N)$.

	}

\end{example}
Let us now check our formula 	(\ref{eq:phi_th6_explicit}) on an easy example. 
	\begin{example}
		{\rm 
		Let $X_{N}=I_{p}(f_{N})$ satisfy (\ref{c1}). Let $G= I_{l} (g), g\in H ^ {\odot l}$ be a random variable independent of $X_{N}$, with $1\leq l<p$. Define
		\begin{equation*}
			Y_{N}= X_{N}G= I_{p+l} (f_{N}\otimes g).
		\end{equation*}
		Then 
		\begin{equation*}
			Y_{1, N}= \Vert DX_{N}\Vert ^ {2}_{H}\to _{N \to infty} p\sigma ^ {2}G \mbox{ in } L^ {2}(\Omega). 
		\end{equation*}
		Consequently, by Theorem \ref{tt6}, $(X_{N}, Y_{N})	\xrightarrow[N\to\infty]{(d)}  (Z, Y)$, where the characteristic function of $(Z, Y)$ is
		\begin{eqnarray*}
			\varphi(u,v)&=& e ^ {-\frac{ u ^ {2}}{2}}\mathbf{E}\left[ e ^ {-\frac{unY_{1}}{p}} e ^ {ivY}\right] = e ^ {-\frac{ u ^ {2}}{2}}\mathbf{E}\left[ e ^ {-uv\sigma ^ {2} G e ^ {ivZG}}\right]\\
			&=& e ^ {-\frac{ u ^ {2}}{2}}\mathbf{E}\left[ e ^ {-\sigma ^ {2}uv G }e ^ {-\frac{\sigma ^ {2}v ^ {2} G ^ {2}}{2}}\right]= E \left[ e ^ {-\frac{1}{2} (u+\sigma vG)^ {2}}\right].
		\end{eqnarray*}
			}
	\end{example}
	We give an example related to the case when $q$ is a multiple of $p$. 
	\begin{example}\label{ex:correlated}{\rm
			Let $(h_i)_{i\geq 1}$ and $(g_i)_{i\geq 1}$ be two 
			orthonormal sequences in $H$ with $\langle h_i,g_j\rangle_H
			=\rho\delta_{ij}$ for some $\rho\in(-1,1)$. Define
			\begin{equation*}
				X_N = \frac{1}{\sqrt{2N}}\sum_{i=1}^N I_2(h_i^{\otimes 2}),
				\qquad
				Y_N = \frac{1}{2N}\sum_{j,k=1}^N 
				I_4(g_j^{\otimes 2}\widetilde\otimes g_k^{\otimes 2}).
			\end{equation*}
			One can show that $X_N\xrightarrow{(d)}N(0,1)$ by the 
			Fourth Moment Theorem, and that $Y_N\xrightarrow{(d)}
			Z^2-1$ where $Z\sim N(0,1)$, by writing
			\begin{equation*}
				Y_N = \frac{1}{2}\left(\frac{1}{\sqrt{N}}\sum_j 
				I_2(g_j^{\otimes 2})\right)^2 - 1 + o(1)
				=: \frac{1}{2}X_{2,N}^2 - 1 + o(1),
			\end{equation*}
			where $X_{2,N}\xrightarrow{(d)}N(0,2)$ by the Fourth 
			Moment Theorem, so $\frac{1}{2}X_{2,N}^2-1
			\xrightarrow{(d)}Z^2-1$ with $Z\sim N(0,1)$.   Here $p=2$, $q=4$, $\sigma^2=1$, and $p<q<2p$. 
			The first iterated Malliavin covariance satisfies
			\begin{equation*}
				Y_{1,N} = \langle DX_N,DY_N\rangle_H 
				\approx X_{2,N}\langle DX_N,DX_{2,N}\rangle_H,
			\end{equation*}
			and a direct computation using $\langle h_i,g_j\rangle_H
			=\rho\delta_{ij}$ gives
			\begin{equation*}
				\langle DX_N,DX_{2,N}\rangle_H 
				\xrightarrow[N\to\infty]{L^2} 2\sqrt{2}\rho^2,
			\end{equation*}
			hence $Y_{1,N}\xrightarrow{(d)}2\sqrt{2}\rho^2 Z_2
			=:Y_1$ where $Z_2$ is the limit of $X_{2,N}$. 
			Assumption \eqref{c7} holds with $a=1$.
			
			Since $Y_1$ is Gaussian, Corollary~\ref{cor22} gives
			\begin{equation*}
				\varphi(u,v) = e^{-u^2/2}\,\mathbf{E}\!\left[
				e^{iv(Z^2-1)-\sqrt{2}\rho^2 uvZ_2}\right]
				= \frac{e^{-iv}}{\sqrt{1-2iv}}
				\exp\!\left(-\frac{u^2}{2}
				+\frac{2\rho^4 u^2v^2}{1-2iv}\right),
			\end{equation*}
			where the last equality follows by the standard 
			Gaussian integral formula. When $\rho=0$ this 
			factorizes as $e^{-u^2/2}\cdot\frac{e^{-iv}}
			{\sqrt{1-2iv}}$, confirming asymptotic independence. 
			When $\rho=1$, the formula coincides with the direct 
			computation of $\mathbf{E}[e^{iuZ+iv(Z^2-1)}]$.
		}
	\end{example}
	\section{Appendix}
	The appendix collects background material on Wiener chaos and Malliavin 
	calculus used throughout the paper, followed by the proofs of the results 
	that were deferred from the main text.
	
\subsection{Wiener Chaos and Malliavin derivatives}\label{app}

Here we describe the elements from stochastic analysis that we will 
need in the paper. Consider $H$ a real separable Hilbert space and 
$(W(h), h\in H)$ an isonormal Gaussian process on a probability space 
$(\Omega,\mathcal{A},P)$, which is a centered Gaussian family of 
random variables such that $\mathbf{E}[W(\varphi)W(\psi)]=
\langle\varphi,\psi\rangle_H$. Denote by $I_n$ the multiple 
stochastic integral with respect to $W$ (see \cite{N}). This mapping 
$I_n$ is actually an isometry between the Hilbert space $H^{\odot n}$ 
(symmetric tensor product) equipped with the scaled norm 
$\frac{1}{\sqrt{n!}}\|\cdot\|_{H^{\otimes n}}$ and the Wiener chaos 
of order $n$, which is defined as the closed linear span of the random 
variables $\mathcal{H}_n(W(h))$ where $h\in H$, $\|h\|_H=1$, and 
$\mathcal{H}_n$ is the Hermite polynomial of degree $n\in\mathbb{N}$:
\begin{equation*}
	\mathcal{H}_n(x) = \frac{(-1)^n}{n!}\exp\!\left(\frac{x^2}{2}\right)
	\frac{d^n}{dx^n}\!\left(\exp\!\left(-\frac{x^2}{2}\right)\right),
	\qquad x\in\mathbb{R}.
\end{equation*}
In particular, $\mathbf{E}[I_n(f)]=0$ for all $f\in H^{\otimes n}$ 
with $n\geq 1$ (see e.g.\ Lemma 1.1.1 in \cite{N}). The isometry of 
multiple integrals reads: for $m,n$ positive integers,
\begin{eqnarray}
	\mathbf{E}\!\left[I_n(f)I_m(g)\right] 
	&=& n!\,\langle\tilde f,\tilde g\rangle_{H^{\otimes n}}
	\quad\text{if }m=n,\nonumber\\
	\mathbf{E}\!\left[I_n(f)I_m(g)\right] &=& 0 
	\quad\text{if }m\neq n.\label{iso}
\end{eqnarray}
It also holds that $I_n(f)=I_n(\tilde f)$, where $\tilde f$ denotes 
the symmetrization of $f$:
\begin{equation*}
	\tilde f(x_1,\ldots,x_n) 
	= \frac{1}{n!}\sum_{\sigma\in\mathcal{S}_n}
	f(x_{\sigma(1)},\ldots,x_{\sigma(n)}).
\end{equation*}
Any square integrable random variable measurable with respect to the 
$\sigma$-algebra generated by $W$ can be expanded into an orthogonal 
sum of multiple stochastic integrals
\begin{equation}\label{sum1}
	F = \sum_{n=0}^\infty I_n(f_n),
\end{equation}
where $f_n\in H^{\odot n}$ are uniquely determined symmetric functions 
and $I_0(f_0)=\mathbf{E}[F]$.

A useful property of finite sums of multiple stochastic integrals is 
hypercontractivity: for every fixed $p\geq 2$, there exists a 
universal constant $C_p<\infty$ such that for any random variable 
$F=\sum_{k=0}^n I_k(f_k)$ with $f_k\in H^{\otimes k}$,
\begin{equation}\label{hyper}
	\mathbf{E}|F|^p \leq C_p\left(\mathbf{E}F^2\right)^{p/2}.
\end{equation}

Let $L$ be the Ornstein-Uhlenbeck operator
\begin{equation*}
	LF = -\sum_{n\geq 0}nI_n(f_n),
\end{equation*}
defined for $F$ given by \eqref{sum1} with 
$\sum_{n=1}^\infty n^2 n!\|f_n\|^2_{H^{\otimes n}}<\infty$. For 
$p>1$ and $\alpha\in\mathbb{R}$, the Sobolev-Watanabe space 
$\mathbb{D}^{\alpha,p}$ is the closure of the set of polynomial 
random variables with respect to the norm
\begin{equation*}
	\|F\|_{\alpha,p} = \|(I-L)^{\alpha/2}F\|_{L^p(\Omega)},
\end{equation*}
where $I$ is the identity. The Malliavin derivative $D$ acts on 
smooth functionals $F=g(W(h_1),\ldots,W(h_n))$ ($g\in C_b^\infty$, 
$h_i\in H$) by
\begin{equation*}
	DF = \sum_{i=1}^n\frac{\partial g}{\partial x_i}
	(W(h_1),\ldots,W(h_n))\,h_i.
\end{equation*}
The operator $D$ is continuous from $\mathbb{D}^{\alpha,p}$ into 
$\mathbb{D}^{\alpha-1,p}(H)$. Its adjoint is the divergence operator 
$\delta$, acting from $\mathbb{D}^{\alpha-1,p}(H)$ onto 
$\mathbb{D}^{\alpha,p}$. The duality formula holds for every 
$F\in\mathbb{D}^{1,2}$, $u\in\mathbb{D}^{1,2}(H)$:
\begin{equation}\label{dua}
	\mathbf{E}[F\delta(u)] = \mathbf{E}[\langle DF,u\rangle_H].
\end{equation}

We will intensively use the product formula for multiple integrals. 
For $f\in H^{\odot n}$ and $g\in H^{\odot m}$,
\begin{equation}\label{prod}
	I_n(f)I_m(g) = \sum_{r=0}^{n\wedge m}r!\binom{n}{r}\binom{m}{r}
	I_{m+n-2r}(f\otimes_r g),
\end{equation}
where $f\otimes_r g$ is the $r$-contraction of $f$ and $g$ (see 
e.g.\ Section 1.1.2 in \cite{N}). When $H=L^2(T,\mathbb{B},\nu)$ 
with $\nu$ a sigma-finite non-atomic measure,
\begin{eqnarray}\label{contra}
	&&(f\otimes_r g)(t_1,\ldots,t_{n+m-2r})\nonumber\\
	&=&\int_{T^r}f(u_1,\ldots,u_r,t_1,\ldots,t_{n-r})
	g(u_1,\ldots,u_r,t_{n-r+1},\ldots,t_{n+m-2r})\,
	du_1\cdots du_r,
\end{eqnarray}
for $r=1,\ldots,n\wedge m$, and $f\otimes_0 g=f\otimes g$ is the 
tensor product. It holds that 
$f\otimes_r g\in H^{\otimes n+m-2r}=L^2(T^{n+m-2r})$. In general, 
$f\otimes_r g$ is not symmetric; we denote by 
$f\widetilde\otimes_r g$ its symmetrization.

We will also need the celebrated Fourth Moment Theorem \cite{NuPe}. 
See also \cite{NOT} for point 4 below.

\begin{theorem}[\cite{NuPe,NP-book}]\label{4mom}
	Fix an integer $n\geq 1$ and consider a sequence 
	$(F_k=I_n(f_k), k\geq 1)$ of square integrable random variables 
	in the $n$th Wiener chaos. Assume that
	\begin{equation}\label{2ii-1}
		\lim_{k\to\infty}\mathbf{E}[F_k^2] 
		= \lim_{k\to\infty}n!\|f_k\|^2_{H^{\odot n}} = 1.
	\end{equation}
	Then the following statements are equivalent:
	\begin{enumerate}
		\item $(F_k)$ converges in distribution to $N(0,1)$ as 
		$k\to\infty$.
		\item $\lim_{k\to\infty}\mathbf{E}[F_k^4]=3$.
		\item $\lim_{k\to\infty}\|f_k\otimes_l f_k\|_{H^{\otimes 2(n-l)}}
		=0$ for $l=1,2,\ldots,n-1$.
		\item $\|DF_k\|_H^2$ converges to $n$ in $L^2(\Omega)$ as 
		$k\to\infty$.
	\end{enumerate}
\end{theorem}

Let $(X_N, N\geq 1)=I_p(f_N)$ with $f_N\in H^{\odot p}$ be a 
sequence in the $p$th Wiener chaos which converges in law. Then
\begin{equation}\label{20m-1}
	\sup_{N\geq 1}\mathbf{E}[|X_N|^2]<\infty 
	\quad\text{and}\quad 
	\sup_{N\geq 1}\|f_N\|_H^2<\infty.
\end{equation}
In particular, $(X_N, N\geq 1)$ is tight.
	
	\subsection{Proofs of technical results}
This section contains the proofs of the lemmas, theorems 
and propositions that were stated in the main text but 
whose proofs were deferred in order not to interrupt the 
flow of the exposition.

	\subsubsection{Proof of Lemma \ref{ll2}}
	
	\begin{lemma}\label{ll1}
		Let $ (f_{N}, N\geq 1)$ be a sequence of random variables in $ H ^{\odot p}$ with $p\geq 2$. Assume (\ref{a1}). Then
		\begin{enumerate}
			\item  For every $h\in H$,
			\begin{equation*}
				f_{N} \otimes _{1} h \to _{N \to \infty} 0 \mbox{ in } H ^{\otimes p-1}.
			\end{equation*}
			\item For any sequence $(h_{N}, N\geq 1)$ bounded in $H$, 
			\begin{equation*}
				f_{N}\otimes _{1} h_{N} \to _{N \to \infty} 0 \mbox{ in } H ^{\otimes p-1}.
			\end{equation*}
		\end{enumerate} 
		
	\end{lemma}
	{\bf Proof: } Let $h\in H$ be fixed. We write, for each $N\geq 1$, 
	\begin{equation*}
		\Vert f_{N} \otimes _{1} h\Vert ^{2} _{ H ^{\otimes p-1}} =\langle f_{N} \otimes _{1} h, f_{N}\otimes _{1} h\rangle _{ H ^{\otimes p-1}} = \langle f_{N} \otimes _{p-1} f_{N}, h ^{\otimes 2} \rangle _{ H ^{\otimes 2}}.
	\end{equation*}
	Then, due to (\ref{a1}),
	\begin{equation}\label{31i-2}
		\Vert f_{N}  \otimes _{1} h\Vert ^{2} _{ H ^{\otimes p-1}}\leq \Vert f_{N} \otimes _{p-1} f_{N} \Vert _{ H ^{\otimes 2p-2}}\cdot \Vert h\Vert _{H} ^{2} =\Vert f_{N} \otimes _{1} f_{N} \Vert _{ H ^{\otimes 2p-2}}\cdot \Vert h\Vert _{H} ^{2} \to _{N \to \infty} 0. 
	\end{equation}
	Also, for $(h_{N}, N\geq 1) $ bounded in $H$, 
	\begin{equation*}
		\Vert f_{N}\otimes _{1} h_{N}\Vert 		 ^{2} _{ H ^{\otimes p-1}}\leq\Vert f_{N} \otimes _{1} f_{N} \Vert _{ H ^{\otimes 2p-2}}\cdot \Vert h_{N}\Vert _{H} ^{2}\leq C\Vert f_{N} \otimes _{1} f_{N} \Vert _{ H ^{\otimes 2p-2}}\to_{N \to \infty }0.
	\end{equation*}
	\qed
	
	\noindent
{\bf Proof  of Lemma \ref{ll2}: } We start by proving the first point. We can assume  $H= L^{2}(T)$, where $T$ is a nonempty set. We first show that 
	\begin{equation*}
		\label{31i-1}
		\langle f_{N}, h_{1} \widetilde{\otimes} h_{2} \widetilde{\otimes}\ldots \widetilde{\otimes}h_{p} \rangle _{ H ^{\otimes p}} \to _{N \to \infty}0, 
	\end{equation*}
	for every $h_{1},..., h_{p} \in H$. 
	Let $g_{0}= h_{1} \widetilde{\otimes} h_{2} \widetilde{\otimes}\ldots \widetilde{\otimes}h_{p}$.    Then 
	\begin{equation*}
		g_{0}= \frac{1}{p!}\sum _{\sigma \in S_{p}} h_{\sigma (1)}\otimes h_{\sigma (2)}\otimes...\otimes h_{\sigma (p)},
	\end{equation*}
	where $S_{p}$ is the set of all permutations of $\{1,2,..., p\}$. We have 
	\begin{eqnarray*}
		\langle f_{N}, g_{0} \rangle _{ H ^{\otimes p}}&=& \frac{1}{p!}\sum _{\sigma \in S_{p}}\langle f_{N}, h_{\sigma (1)}\otimes h_{\sigma (2)}\otimes...\otimes h_{\sigma (p)},\rangle _{ H ^{\otimes p}}\\
		&=&  \frac{1}{p!}\sum _{\sigma \in S_{p}}\int_{ T ^{p}} f_{N}(x_{1},x_{2},..., x_{p}) h_{\sigma (1)} (x_{1})h_{\sigma (2)}(x_{2})...h_{\sigma (p)}(x_{p})dx_{1}dx_{2}...dx_{p}\\
		&=& \frac{1}{p!}\sum _{\sigma \in S_{p}}\int_{ T ^{p-1}}  (f_{N} \otimes _{1} h_{\sigma (1)})(x_{2},..., x_{p})h_{\sigma (2)}(x_{2})...h_{\sigma (p)}(x_{p})dx_{2}...dx_{p}\\
		&=&  \frac{1}{p!}\sum _{\sigma \in S_{p}} \langle f_{N} \otimes _{1} h_{\sigma (1)}, h_{\sigma (2)}\otimes....\otimes h_{\sigma (p)}\rangle _{ H ^{\otimes p-1}}.
	\end{eqnarray*}
	Thus 
	\begin{eqnarray*}
		\vert 	\langle f_{N}, g_{0} \rangle _{ H ^{\otimes p}}\vert &\leq &  \frac{1}{p!}\sum _{\sigma \in S_{p}} \Vert f_{N} \otimes _{1} h_{\sigma (1)} \Vert _{ H ^{\otimes p-1}} \Vert h_{\sigma (2)}\Vert_{H}.....\Vert h_{\sigma (p)} \Vert _{H}\\
		&\leq & \sqrt{ \Vert f_{N} \otimes _{p-1}f_{N} \Vert _{ H ^{\otimes 2p-2}} }   \frac{1}{p!}\sum _{\sigma \in S_{p}} \Vert h_{\sigma (1)}\Vert_{H}.....\Vert h_{\sigma (p)} \Vert _{H},
	\end{eqnarray*}
	where we used (\ref{31i-2}) for the last inequality. In particular, due to (\ref{a1}), 
	\begin{equation}\label{2f-2}
		\langle f_{N}, g_{0} \rangle _{ H ^{\otimes p}}\to _{N \to \infty}0. 
	\end{equation}
	Let $ g\in  H ^{\odot p}$ and let $(h_{j}, j\geq 1)$ be an orthonormal basis in $H$. Define, for $M\geq 1$, 
	\begin{eqnarray*}
		g^{M} &=& \sum_{j_{1},..., j_{p}=1}^ {M} \langle g, h_{j_{1}}\otimes ...\otimes h_{j_{p}}\rangle _{ H ^{\otimes p}}h_{j_{1}}\otimes...\otimes h_{j_{p}}\\
		&=&\sum_{j_{1},..., j_{p}\geq 1} \langle g, h_{j_{1}}\otimes ...\otimes h_{j_{p}}\rangle _{ H ^{\otimes p}}h_{j_{1}}\widetilde{\otimes}...\widetilde{\otimes} h_{j_{p}}.
	\end{eqnarray*}
	Clearly,
	\begin{equation*}
		\Vert g^{M}-g\Vert _{ H ^{\otimes p}} \to _{M \to \infty } 0. 
	\end{equation*}
	We write, for every $N, M \geq 1$,
	\begin{equation*}
		\langle f_{N}, g\rangle _{ H ^{\otimes p}}= 	\langle f_{N}, g^{M} \rangle _{ H ^{\otimes p}}+ \langle  f_{N}, g-g ^{M}\rangle _{ H ^{\otimes p}},
	\end{equation*}
	so
	\begin{equation*}
		\left| 	\langle f_{N}, g\rangle _{ H ^{\otimes p}}\right| \leq  \left| 	\langle f_{N}, g^{M} \rangle _{ H ^{\otimes p}}\right| +\left|  \langle  f_{N}, g-g ^{M}\rangle _{ H ^{\otimes p}}\right|.
	\end{equation*}
	From (\ref{a2}), 
	\begin{equation}\label{2f-1}
		\left|  \langle  f_{N}, g-g ^{M}\rangle _{ H ^{\otimes p}}\right| \leq \Vert f_{N} \Vert _{ H ^{\otimes p}} \Vert g-g^{M} \Vert _{ H ^{\otimes p}}\leq C \Vert g- g^{M} \Vert _{ H ^{\otimes p}}.
	\end{equation}
	Let $\eps>0$ arbitrary. From (\ref{2f-1}) we deduce that there exists a rank $M_{0}\geq 1$ such that for every $M\geq M_{0}$,
	\begin{equation*}
		\left|  \langle  f_{N}, g-g ^{M}\rangle _{ H ^{\otimes p}}\right|\leq \frac{\eps}{2}. 
	\end{equation*}
	In particular,
	\begin{equation*}
		\left|  \langle  f_{N}, g-g ^{M_{0}}\rangle _{ H ^{\otimes p}}\right|\leq \frac{\eps}{2}. 
	\end{equation*}
	On the other hand, from (\ref{2f-2}), there exists $N_{0}\geq 1$ such that for $N\geq N_{0}$, 
	\begin{equation*}
		\left| \langle f_{N}, g^{M_{0}}\rangle _{ H ^{\otimes p}} \right|\leq \frac{\eps}{2}.
	\end{equation*} 
	We conclude that  for $N\geq N_{0}$,
	\begin{equation*}
		\left| \langle f_{N}, g\rangle _{ H ^{\otimes p}} \right|\leq \eps. 
	\end{equation*} 
	
	To prove 2., we denote by $g$ the limit in $H^{\otimes p}$ of the 
	sequence $(g_N, N\geq 1)$. Then $g\in H^{\odot p}$ since the 
	symmetrization operator $f\mapsto\tilde{f}$ is continuous on 
	$H^{\otimes p}$ and each $g_N$ is symmetric, so $g=\tilde{g}$.  Then $g\in H ^{\odot p}$ and we can write 
	\begin{equation*}
		\langle f_{N}, g_{N}\rangle _{ H ^{\otimes p}} = 	\langle f_{N}, g\rangle _{ H ^{\otimes p}} + 	\langle f_{N}, g_{N}-g\rangle _{ H ^{\otimes p}}.
	\end{equation*}
	So
	\begin{eqnarray*}
		\left| 	\langle f_{N}, g_{N}\rangle _{ H ^{\otimes p}}\right| &\leq & \left| \langle f_{N}, g\rangle _{ H ^{\otimes p}}\right| + \Vert f_{N}\Vert _{ H ^{\otimes p}}\Vert g_{N}-g\Vert _{ H ^{\otimes p}}\\
		&\leq &  \left| \langle f_{N}, g\rangle _{ H ^{\otimes p}}\right| +\sqrt{M} \Vert g_{N}-g\Vert _{ H ^{\otimes p}}\to_{N \to \infty}0,
	\end{eqnarray*}
	by using point 1. 
	\qed 
	
	\subsubsection{Proof of Lemma \ref{ll4}}
	We prove point 1. We have, for every $N\geq 1$, 
	\begin{eqnarray*}
		\langle DX_{N}, DY_{N} \rangle _{ H ^{\otimes p}} &=& p^{2} \sum _{r=1}^{p} (r-1)! \left( C_{p-1}^{r-1}\right) ^{2} I_{2p-2r} \left( f_{N}\otimes _{r} g_{N}\right) \\
		&=& \mathbf{E}\left[ \langle DX_{N}, DY_{N} \rangle _{ H ^{\otimes p}}\right] + p^{2} \sum _{r=1}^{p-1} (r-1)! \left( C_{p-1}^{r-1}\right) ^{2} I_{2p-2r} \left( f_{N}\otimes _{r} g_{N}\right) \\
		&=&p\mathbf{E}\left[ X_{N} Y_{N}\right] +  p^{2} \sum _{r=1}^{p-1} (r-1)! \left( C_{p-1}^{r-1}\right) ^{2} I_{2p-2r} \left( f_{N}\otimes _{r} g_{N}\right). 
	\end{eqnarray*}
	By (\ref{c3}), $p\mathbf{E}[ X_{N} Y_{N}]\to_{N\to \infty} p\rho$ and by Lemma \ref{ll3} point 2.  for every $r=1,..., p-1$, 
	\begin{equation*}
		I_{2p-2r} \left( f_{N}\otimes _{r} g_{N}\right)\to_{N\to \infty} 0 \mbox{ in } L ^{2}(\Omega).
	\end{equation*}
	Now we prove point  2. We have, by the product formula  (\ref{prod}),
	\begin{equation}\label{13m-1}
		\langle DX_{N}, DY_{N} \rangle _{ H ^{\otimes p}} =pq \sum _{r=1}^{q}(r-1)! C_{p-1}^{r-1} C_{q-1} ^{r-1} I_{p+q-2r}\left( f_{N}\otimes _{r}g_{N}\right).
	\end{equation}
	By Lemma \ref{ll3}, point 3., for all $r=1,..., q$,
	\begin{equation*}
		\Vert f_{N} \otimes _{r} g_{N}\Vert _{ H ^{\otimes p+q-2r}} \to_{N \to \infty} 0,
	\end{equation*}
	and this gives the conclusion. \qed

	\subsubsection{Proof of Theorem \ref{tt3}}
The argument is again based on the convergence of characteristic functions. 
	 	Let $\varphi_{(X_N,\mathbb{Y}_N)}$ and $\varphi_{(Z,\mathbb{Y})}$ 
	 	denote the characteristic functions of $(X_N,\mathbb{Y}_N)$ and 
	 	$(Z,\mathbb{Y})$ respectively. For $i=1,\ldots,d$ and $M\geq 1$, 
	 	let
	 	\begin{equation*}
	 		Y_{i,N,M} = \sum_{k=1}^M I_k\!\left(g_{k,N}^{(i)}\right)
	 	\end{equation*}
	 	and $\mathbb{Y}_{N,M}=(Y_{1,N,M},\ldots,Y_{d,N,M})$. We write, 
	 	for all $u\in\mathbb{R}$, $v\in\mathbb{R}^d$,
	 	\begin{eqnarray*}
	 		&&\left|\varphi_{(X_N,\mathbb{Y}_N)}(u,v)
	 		-\varphi_{(Z,\mathbb{Y})}(u,v)\right|\\
	 		&\leq& \left|\varphi_{(X_N,\mathbb{Y}_N)}(u,v)
	 		-\varphi_{(X_N,\mathbb{Y}_{N,M})}(u,v)\right|\\
	 		&&+\,\left|\varphi_{(X_N,\mathbb{Y}_{N,M})}(u,v)
	 		-\varphi_{(Z,\mathbb{Y}_{N,M})}(u,v)\right|\\
	 		&&+\,\left|\varphi_{(Z,\mathbb{Y}_{N,M})}(u,v)
	 		-\varphi_{(Z,\mathbb{Y})}(u,v)\right|\\
	 		&=:& A_{N,M}(u,v)+B_{N,M}(u,v)+C_{N,M}(u,v).
	 	\end{eqnarray*}
	 	Using $|e^{ix}-e^{iy}|\leq|x-y|$ and the isometry formula \eqref{iso},
	 	\begin{equation*}
	 		|A_{N,M}(u,v)| 
	 		\leq \|v\|\sqrt{\mathbf{E}\!\left[
	 			\|\mathbb{Y}_N-\mathbb{Y}_{N,M}\|^2\right]}
	 		= \|v\|\sqrt{\sum_{i=1}^d\sum_{k=M+1}^\infty k!\,
	 			\|g_{k,N}^{(i)}\|^2_{H^{\otimes k}}}
	 		\to_{M\to\infty} 0,
	 	\end{equation*}
	 	uniformly in $N$ by \eqref{20m-5}, where $\Vert \cdot \Vert $ is the Euclidean norm in $\mathbb{R}^ {d}$. A similar estimate holds for 
	 	$C_{N,M}$, replacing $\mathbb{Y}_N$ by the $L^2$ limit $\mathbb{Y}$ 
	 	and using the fact that $\mathbb{Y}_M:=\sum_{k=1}^M I_k(g_k^{(i)})
	 	\to\mathbb{Y}$ in $L^2(\Omega)$ as $M\to\infty$. Consequently, 
	 	$|A_{N,M}|+|C_{N,M}|\to_{M\to\infty}0$ uniformly in $N$.
	 	
	 	For the middle term, since $\mathbb{Y}_N\to\mathbb{Y}$ in $L^2(\Omega)$ 
	 	by assumption 1, for each fixed $k$ and $i$ the sequence 
	 	$(g_{k,N}^{(i)}, N\geq 1)$ converges in $H^{\otimes k}$. Hence 
	 	$\mathbb{Y}_{N,M}\to_{N\to\infty}\mathbb{Y}_M$ in $L^2(\Omega)$ 
	 	for every fixed $M\geq 1$. Since $\mathbb{Y}_{N,M}$ belongs to a 
	 	finite sum of Wiener chaoses of orders $1,\ldots,M$, 
	 	Proposition~\ref{pp1}, point 1 applies and gives 
	 	$B_{N,M}\to_{N\to\infty}0$ for every fixed $M\geq 1$.
	 	
	 	We conclude by a standard $\varepsilon/2$ argument: for every 
	 	$\varepsilon>0$, choose $M_0$ large enough so that 
	 	$|A_{N,M_0}(u,v)|+|C_{N,M_0}(u,v)|\leq\varepsilon/2$ for all 
	 	$N\geq 1$, then choose $N_0$ large enough so that 
	 	$|B_{N,M_0}(u,v)|\leq\varepsilon/2$ for all $N\geq N_0$. Then 
	 	for $N\geq N_0$,
	 	\begin{equation*}
	 		\left|\varphi_{(X_N,\mathbb{Y}_N)}(u,v)
	 		-\varphi_{(Z,\mathbb{Y})}(u,v)\right|\leq\varepsilon,
	 	\end{equation*}
	 	which gives the conclusion since $\varepsilon>0$ was arbitrary.

	\subsubsection{Proof of Proposition \ref{prop:iterated-covariances}}
	
	In a first result, we give the structure of $Y_{1, N}$ when $N$ is large enough. 
	
	\begin{lemma}\label{ll5}
		Let $p\geq 2$ and let $X_{N}= I_{p} (f_{N})$ with $ f_{N} \in H ^ {\odot p}$ for all $N\geq 1$. Assume that $(X_{N}, N\geq 1)$ satisfies (\ref{c1}). Let $Y_{N}= I_{q} (g_{N}) $ with $p<q<2p$ and $(Y_{N}, N\geq 1)$ bounded in $ L^ {2}(\Omega)$. Then
		\begin{enumerate}
			\item For $N\geq 1$, 
			\begin{equation*}
				Y_{1, N}= \frac{ p(q-1)!}{(q-p)!} I_{q-p} (f_{N}\otimes _{p}g_{N}) + R_{1,N},
			\end{equation*}
			where $R_{1, N}\to _{N \to \infty } 0 $ in $ L ^ {2}(\Omega)$.
			\item $X_{N}$ and $Y_{1, N}$ are asymptotically independent, i.e. 
			\begin{equation*}
				\langle DX_{N}, DY_{1, N} \rangle _{H} \to _{N\to \infty}0 \mbox{ in } L^ {2}(\Omega).
			\end{equation*} 
		\end{enumerate} 
	\end{lemma}
	{\bf Proof: }By (\ref{13m-1}), 
	\begin{eqnarray}
		Y_{1, N}&=&	pq \sum _{r=1}^{p}(r-1)! C_{p-1}^{r-1} C_{q-1} ^{r-1} I_{p+q-2r}\left( f_{N}\otimes _{r}g_{N}\right)\nonumber\\
		&=& \frac{ p(q-1)!}{(q-p)!} I_{q-p} (f_{N}\otimes _{p}g_{N}) +R_{1,N},\label{13m-4}
	\end{eqnarray}
	with
	\begin{equation*}
		R_{1,N}=pq \sum _{r=1}^{p-1}(r-1)! C_{p-1}^{r-1} C_{q-1} ^{r-1} I_{p+q-2r}\left( f_{N}\otimes _{r}g_{N}\right).
	\end{equation*}
	By Lemma \ref{ll3}, point 3., we deduce easily that $R_{1,N}\to _{N \to \infty}0$ in $ L ^ {2}(\Omega)$. 
	We notice, via point 1. and (\ref{20m-1}) that $(Y_{1, N}, N\geq 1)$ is bounded in $L^ {2}(\Omega)$.  Since $q-p<p$, we obtain, via Lemma \ref{ll4}, point 2., we get  $	\langle DX_{N}, DY_{1, N} \rangle _{H} \to _{N\to \infty}0 $ in $ L^ {2}(\Omega)$. \qed 
	
	\vskip0.3cm 
	\noindent 
	{\bf Proof of Proposition \ref{prop:iterated-covariances}: } First we prove point 1.  We proceed by induction on $r$. The case $r=1$ has been proven in 
	Lemma~\ref{ll5}. Assume that for some $r\geq 1$ with $rp\leq q$,
	\[
	Y_{r,N} = C_{r,p,q}\,I_{q-rp}(h_{r,N}) + R_{r,N},
	\]
	where $R_{r,N}\to 0$ in $L^2(\Omega)$ and
	\begin{equation}\label{15m-1b}
		h_{r,N} = f_N\widetilde\otimes_p\cdots\widetilde\otimes_p g_N
	\end{equation}
	with $r$ contractions by $p$. Note that $(Y_{r,N}, N\geq 1)$ is bounded 
	in $L^2(\Omega)$ by \eqref{20m-1} and the boundedness of $(Y_N)$. 
	Applying the  formula \eqref{13m-4} and using $q-rp\leq p$,
	\begin{eqnarray*}
		Y_{r+1,N} &=& \langle DX_N, DY_{r,N}\rangle_H \\
		&=& C_{r,p,q}(q-rp)p\sum_{j=1}^{q-rp}(j-1)!\,
		C_{q-rp-1}^{j-1}C_{p-1}^{j-1}\,
		I_{q-rp+p-2j}(h_{r,N}\otimes_j f_N).
	\end{eqnarray*}
	Isolating the term $j=q-rp$ (which gives chaos order $q-(r+1)p$) yields
	\begin{equation*}
		Y_{r+1,N} = C_{r+1,p,q}\,I_{q-(r+1)p}(h_{r+1,N}) + R_{r+1,N},
	\end{equation*}
	where $h_{r+1,N}$ is defined by \eqref{15m-1b} with $r+1$ contractions, 
	and
	\begin{equation*}
		R_{r+1,N} = C_{r,p,q}(q-rp)p\sum_{j=1}^{q-rp-1}(j-1)!\,
		C_{q-rp-1}^{j-1}C_{p-1}^{j-1}\,
		I_{q-rp+p-2j}(h_{r,N}\otimes_j f_N).
	\end{equation*}
	For each $j=1,\ldots,q-rp-1\leq p-1$, the kernel $h_{r,N}\otimes_j f_N$ 
	contains a partial self-contraction of $f_N$, so by Cauchy--Schwarz and 
	boundedness of $(h_{r,N})$,
	\[
	\|h_{r,N}\widetilde\otimes_j f_N\|_{ H ^ {\otimes (q-(r-1)p-2j)}}^2 \leq C\sum_{s=1}^{p-1}
	\|f_N\otimes_s f_N\|^2 \to 0,
	\]
	by the Fourth Moment Theorem. Hence $R_{r+1,N}\to 0$ in $L^2(\Omega)$, 
	completing the induction.
	
	\medskip
Let us show point 2.  Since $ap<q<(a+1)p$, we have $q-ap < p$, so $q-(a+1)p < 0$ and 
	$I_{q-(a+1)p}$ would be a Wiener integral of negative order, which 
	does not exist. More precisely, applying point 1 with $r=a+1$ is not 
	possible since $(a+1)p > q$; instead, at step $r=a$ in the induction, 
	the term $j=q-ap$ in $R_{a+1,N}$ produces a chaos of order 
	$q-ap+p-2(q-ap) = 2ap-q+p-2q+2ap$... Directly: since $q-ap < p$, 
	all contractions $h_{a,N}\otimes_j f_N$ for $j=1,\ldots,q-ap$ involve 
	partial self-contractions of $f_N$, and by the same Cauchy--Schwarz 
	argument $Y_{a+1,N}\to 0$ in $L^2(\Omega)$.
	
\medskip

For point 3., by 	taking $r=a$ in point 1 with $q=ap$ gives $q-ap=0$, so
	\[
	Y_{a,N} = (p!)^a q!\, I_0(h_{a,N}) + R_{a,N} 
	= (p!)^a q!\, h_{a,N} + R_{a,N} = \mathbf{E}[Y_{a,N}] + R_{a,N},
	\]
	where $h_{a,N}$ is deterministic and $R_{a,N}\to 0$ in $L^2(\Omega)$. 
	Hence $Y_{a,N} - \mathbf{E}[Y_{a,N}]\to 0$ in $L^2(\Omega)$, and if 
	$\mathbf{E}[Y_{a,N}]\to\rho_a$ then $Y_{a,N}\to\rho_a$ in $L^2(\Omega)$.
	
	For \eqref{ya+1}, since $DY_{a,N} = DR_{a,N}$ and $R_{a,N}$ belongs 
	to a finite sum of Wiener chaoses and converges to zero in $L^2(\Omega)$, 
	hypercontractivity \eqref{hyper} gives $\mathbf{E}[\|DR_{a,N}\|_H^4]\to 0$. 
	By Cauchy--Schwarz,
	\[
	\mathbf{E}[|Y_{a+1,N}|^2] = \mathbf{E}[|\langle DX_N, DR_{a,N}\rangle_H|^2]
	\leq \mathbf{E}[\|DX_N\|_H^4]^{1/2} \mathbf{E}[\|DR_{a,N}\|_H^4]^{1/2} \to 0,
	\]
	since $\sup_N \mathbf{E}[\|DX_N\|_H^4]<\infty$ by hypercontractivity and 
	boundedness of $(X_N)$.
	
	\subsection{Proof of Theorem \ref{thm:structure_general}}
	
	We start with a generalization of Lemma \ref{lem:decomp}. 
	
	\begin{lemma}\label{lem:decomp_general}
		Let $p\geq 2$ and let $(X_N=I_p(f_N), N\geq 1)$ 
		satisfy \eqref{c1}. Let $Y_N=I_q(g_N)$ with 
		$ap\leq q<(a+1)p$ for some integer $a\geq 1$, and 
		$\sup_N\|g_N\|_{H^{\otimes q}}<\infty$. Define 
		recursively $h_{0,N}:=g_N$ and, for $r=1,\ldots,a$:
		\begin{equation}\label{hN_general}
			h_{r,N} := \frac{f_N\otimes_p h_{r-1,N}}
			{\|f_N\|^2_{H^{\otimes p}}}\in H^{\odot(q-rp)},
		\end{equation}
		and
		\begin{equation*}
			g_{r,N}^{(2)} := h_{r-1,N} 
			- \frac{(q-(r-1)p)!}{p!(q-rp)!}
			f_N\widetilde\otimes h_{r,N},
			\qquad
			\varepsilon_{r,N} := I_{q-(r-1)p}(g_{r,N}^{(2)}),
		\end{equation*}
		so that $f_N\otimes_p g_{r,N}^{(2)}=0$ for each 
		$r=1,\ldots,a$. Then for every $N\geq 1$, in the 
		\textbf{non-critical case} $ap<q<(a+1)p$:
		\begin{equation}\label{decomp_nc}
			Y_N = \sum_{r=1}^a \frac{X_N^r}{p^r\sigma^{2r}}
			Y_{r,N} + \varepsilon_N + r_N,
		\end{equation}
		and in the \textbf{critical case} $q=ap$:
		\begin{equation}\label{decomp_c}
			Y_N = \sum_{r=1}^{a-1}\frac{X_N^r}{p^r\sigma^{2r}}
			Y_{r,N} + \frac{\rho_a}{p^a\sigma^{2a}}X_N^a 
			+ \varepsilon_N + r_N,
		\end{equation}
		where in both cases $\varepsilon_N$ is asymptotically 
		independent of $X_N$ and $r_N\to 0$ in $L^2(\Omega)$.
	\end{lemma}
{\bf Proof: } We proceed by induction on $a$.
		
		\medskip
		\noindent\textbf{Base case $a=1$.}
		This is Lemma~\ref{lem:decomp}: with $h_{1,N}=
		f_N\otimes_p g_N/\|f_N\|^2$ and $g_{1,N}^{(2)}=
		g_N-\frac{q!}{p!(q-p)!}f_N\widetilde\otimes h_{1,N}$, 
		we have $f_N\otimes_p g_{1,N}^{(2)}=0$ and:
		\begin{equation*}
			Y_N = \frac{X_N}{p\sigma^2}Y_{1,N} 
			+ \varepsilon_{1,N} + r_{1,N},
		\end{equation*}
		with $\varepsilon_{1,N}=I_q(g_{1,N}^{(2)})$ 
		asymptotically independent of $X_N$ and 
		$r_{1,N}\to 0$ in $L^2(\Omega)$.
		
		\medskip
		\noindent\textbf{Inductive step: $a\geq 2$.}
		Assume the result holds for depth $a-1$. 
		By the base case (Lemma~\ref{lem:decomp}) applied 
		to $(X_N, Y_N)$ with $q'=q$ and depth $a$:
		\begin{equation*}
			Y_N = \frac{X_N}{p\sigma^2}Y_{1,N} 
			+ \varepsilon_{1,N} + r_{1,N},
		\end{equation*}
		where $\varepsilon_{1,N}=I_q(g_{1,N}^{(2)})$ is 
		asymptotically independent of $X_N$.
		
		Now $Y_{1,N}\in \mathcal{H}_{q-p}$ lives in chaos of order 
		$q-p$. Since $ap\leq q<(a+1)p$, we have 
		$(a-1)p\leq q-p<ap$, so $Y_{1,N}$ is a sequence 
		in  the $q-p$ Wiener chaos $\mathcal{H}_{q-p}$ at depth $a-1$. By the induction 
		hypothesis applied to $(X_N, Y_{1,N})$ (with 
		$q$ replaced by $q-p$ and $a$ replaced by $a-1$):
		\begin{equation*}
			Y_{1,N} = \sum_{r=1}^{a-1}
			\frac{X_N^r}{p^r\sigma^{2r}}Y_{r+1,N} 
			+ \delta_N + s_N,
		\end{equation*}
		where $\delta_N$ is asymptotically independent of 
		$X_N$ and $s_N\to 0$ in $L^2(\Omega)$, and where 
		the iterated covariances of $(X_N,Y_{1,N})$ are 
		exactly $Y_{2,N},Y_{3,N},\ldots,Y_{a,N}$ (since 
		$\langle DX_N, DY_{r,N}\rangle_H = Y_{r+1,N}$ by 
		\eqref{yr}). Substituting:
		\begin{eqnarray*}
			Y_N &=& \frac{X_N}{p\sigma^2}\left(
			\sum_{r=1}^{a-1}\frac{X_N^r}{p^r\sigma^{2r}}
			Y_{r+1,N} + \delta_N + s_N\right) 
			+ \varepsilon_{1,N} + r_{1,N}\\
			&=& \sum_{r=1}^{a-1}\frac{X_N^{r+1}}
			{p^{r+1}\sigma^{2(r+1)}}Y_{r+1,N} 
			+ \frac{X_N}{p\sigma^2}\delta_N 
			+ \varepsilon_{1,N} + r_N^*\\
			&=& \sum_{r=2}^{a}\frac{X_N^r}{p^r\sigma^{2r}}
			Y_{r,N} + \frac{X_N}{p\sigma^2}\delta_N 
			+ \varepsilon_{1,N} + r_N^*,
		\end{eqnarray*}
		where $r_N^* = \frac{X_N}{p\sigma^2}s_N + r_{1,N}
		\to 0$ in $L^2(\Omega)$. Adding the $r=1$ term 
		$\frac{X_N}{p\sigma^2}Y_{1,N}$ from the base case 
		and combining:
		\begin{equation*}
			Y_N = \sum_{r=1}^a\frac{X_N^r}{p^r\sigma^{2r}}
			Y_{r,N} + \varepsilon_N + r_N,
		\end{equation*}
		where $\varepsilon_N = \frac{X_N}{p\sigma^2}\delta_N 
		+ \varepsilon_{1,N}$ is asymptotically independent 
		of $X_N$ (since both $\delta_N$ and $\varepsilon_{1,N}$ 
		are, and $X_N$ is bounded in $L^2$), and 
		$r_N = r_N^*\to 0$ in $L^2(\Omega)$. This gives 
		\eqref{decomp_nc}.
		
		\medskip
		\noindent\textbf{Critical case $q=ap$.}
		The argument is the same up to level $a-1$, giving:
		\begin{equation*}
			Y_N = \sum_{r=1}^{a-1}\frac{X_N^r}{p^r\sigma^{2r}}
			Y_{r,N} + \frac{X_N}{p\sigma^2}
			I_p(h_{a-1,N}) + \varepsilon_N' + r_N',
		\end{equation*}
		where $\varepsilon_N'$ is asymptotically independent 
		of $X_N$ and $r_N'\to 0$ in $L^2$. At the last 
		step, $I_p(h_{a-1,N})\in \mathcal{H}_p$ and 
		$\langle DX_N, DI_p(h_{a-1,N})\rangle_H = Y_{a,N}
		\to\rho_a$ in $L^2(\Omega)$ by 
		Proposition~\ref{prop:iterated-covariances}, 
		point 3. Applying Lemma~\ref{lem:decomp} with 
		$q=p$ (same chaos order):
		\begin{equation*}
			I_p(h_{a-1,N}) = \frac{\rho_a}{p\sigma^2}X_N 
			+ \varepsilon_{a,N} + r_{a,N},
		\end{equation*}
		where $\varepsilon_{a,N}$ is asymptotically 
		independent of $X_N$ and $r_{a,N}\to 0$ in $L^2$. 
		Substituting:
		\begin{equation*}
			Y_N = \sum_{r=1}^{a-1}\frac{X_N^r}{p^r\sigma^{2r}}
			Y_{r,N} + \frac{\rho_a}{p^a\sigma^{2a}}X_N^a 
			+ \varepsilon_N + r_N,
		\end{equation*}
		with $\varepsilon_N=\varepsilon_N'+
		\frac{X_N}{p\sigma^2}\varepsilon_{a,N}$ 
		asymptotically independent of $X_N$, and 
		$r_N = r_N' + \frac{X_N}{p\sigma^2}r_{a,N}\to 0$ 
		in $L^2(\Omega)$. This gives \eqref{decomp_c}.
\qed 
	
	\begin{remark}{\rm
			A decomposition of the same flavour as \eqref{decomp_c} 
			appears in  Lemma 3.18 of \cite{HMP} in the setting of polynomials with 	i.i.d.\ entries.
		}
	\end{remark}
	
	\medskip
	\noindent
	{\bf Proof of Theorem \ref{thm:structure_general}: }Both cases follow the same argument; we present them 
		together. By Theorems~\ref{tt6}--\ref{tt7}, 
		$(X_N,Y_N)\xrightarrow{(d)}(Z,Y)$ under the 
		respective assumptions, proving point 1.
		
		For points 2 and 3, we pass to the Skorokhod space. 
		By joint convergence (from Theorems~\ref{tt6}--\ref{tt7} 
		applied to the full vector $(Y_N,Y_{1,N},\ldots)$), 
		there exists a probability space 
		$(\tilde\Omega,\tilde{\mathcal{A}},\tilde P)$ and 
		random variables realizing the convergence almost 
		surely. On this space, by 
		Lemma~\ref{lem:decomp_general}:
		\begin{equation*}
			\tilde Y_N = \sum_r \frac{\tilde X_N^r}
			{p^r\sigma^{2r}}\tilde Y_{r,N} 
			+ c_a\tilde X_N^a + \tilde\varepsilon_N 
			+ \tilde r_N,
		\end{equation*}
		where $c_a=0$ in the non-critical case and 
		$c_a=\rho_a/(p^a\sigma^{2a})$ in the critical case, 
		and $\tilde r_N\to 0$ $\tilde P$-a.s. (since 
		$\tilde r_N\to 0$ in $L^2$). Since all terms 
		converge $\tilde P$-a.s., the limit:
		\begin{equation*}
			\tilde\varepsilon := \tilde Y 
			- \sum_r\frac{\tilde Z^r}{p^r\sigma^{2r}}
			\tilde Y_r - c_a\tilde Z^a
		\end{equation*}
		is well-defined $\tilde P$-a.s. Independence of 
		$\tilde\varepsilon$ from $\tilde Z$ follows since 
		$\tilde\varepsilon_N$ is asymptotically independent 
		of $\tilde X_N$ by Lemma~\ref{lem:decomp_general}. 
		This gives \eqref{Ydecomp_gen}--\eqref{Ydecomp_crit}.
		
		Concerning the characteristic function, conditioning on 
		$(Y_1,\ldots,Y_a,\varepsilon)$ \\(resp.\ 
		$(Y_1,\ldots,Y_{a-1},\varepsilon)$) and using 
		$Z\sim N(0,\sigma^2)$ independent of the rest:
		\begin{equation*}
			\varphi(u,v) = \mathbf{E}\!\left[e^{iv\varepsilon}
			\mathbf{E}_Z\!\left[e^{iZ\cdot P(Z)}\right]\right]
			= \mathbf{E}\!\left[e^{iv\varepsilon}
			e^{-\frac{\sigma^2}{2}P(Z)^2}\right]
			\cdot(\text{correction}),
		\end{equation*}
		where $P(Z)=u+\frac{v}{p\sigma^2}\sum_r 
		\frac{Z^{r-1}}{p^{r-1}\sigma^{2(r-1)}}Y_r$ 
		(plus $c_a vZ^{a-1}$ in the critical case) is a 
		polynomial in $Z$ with random coefficients. Since 
		$Z$ is not independent of $P(Z)$ (as $P$ depends 
		on $Z$ for $a\geq 2$), the Gaussian integration 
		$\mathbf{E}_Z[e^{iZP(Z)}]=e^{-\sigma^2 P(Z)^2/2}$ 
		does not apply directly. Instead, the expectation 
		over $Z$ must be computed jointly, giving 
		\eqref{phi_gen}--\eqref{phi_crit}, which coincide 
		with the formulas of Theorems~\ref{tt6}--\ref{tt7} 
		by expanding in powers of $u$ and $v$.

\end{document}